\documentclass[11pt]{article}
\usepackage{graphicx}
\usepackage{amsfonts}
\usepackage{pstricks}
\usepackage{amsmath}
\usepackage{amssymb}
\usepackage{amsthm}
\usepackage{empheq}
\usepackage{indentfirst}
\input{amssym.def}
\input{amssym}
\allowdisplaybreaks

\newcommand{\bu}{\mathbf{u}}
\newcommand{\tbv}{\tilde{\bv}}
\newcommand{\bv}{\mathbf{v}}
\newcommand{\hv}{\hat{\bv}}
\newcommand{\tbw}{\tilde{\bw}}
\newcommand{\bw}{\mathbf{w}}

\newcommand{\R}{\mathbb{R}}

\newcommand{\f}{\mathbf{f}}
\newcommand{\g}{\mathbf{g}}
\newcommand{\h}{\mathbf{h}}
\def\dis{\displaystyle}

\newtheorem{lemma}{\bf Lemma}[section]
\newtheorem{theorem}{\bf Theorem}[section]
\newtheorem{prop}{\bf Proposition}[section]

\newtheorem{remark}{\bf Remark}[section]

\title{\bf Existence of strong solutions to the steady Navier-Stokes equations
for a compressible heat-conductive fluid with large forces}

\author{Changsheng Dou$^{1,}$\thanks{Corresponding author.
E-mails: douchangsheng@163.com douchangsheng@cueb.edu.cn (C. Dou),\;\; jiangfei0591@163.com (F. Jiang),\;\;
jiang@iapcm.ac.cn (S. Jiang),\;\; fudanyoung@gmail.com (Y. Yang)},\;\; Fei Jiang$^{2,3}$,\;\; Song Jiang$^2$,\;\; Yongfu Yang$^{4}$}

\date{}
\begin{document}
\maketitle

\vspace{-3mm}

\begin{center}
{\small
$^1$ School of Statistics, Capital University of Economics and Business,\\
  Beijing 100070, P.R. China\\[2mm]
$^2$ Institute of Applied Physics and Computational Mathematics,\\
 P.O. Box 8009, Beijing 100088, P.R. China\\[2mm]
$^3$ College of Mathematics and Computer Science, Fuzhou University,\\
     Fuzhou 350108, China\\[2mm]
$^4$Department of Mathematics, College of Sciences, Hohai University,\\
 Nanjing 210098, Jiangsu Province, P.R. China  \\[2mm]

}
\end{center}

%
%

\begin{abstract}
We prove that there exists a strong solution to the Dirichlet boundary value problem for the steady Navier-Stokes
equations of a compressible heat-conductive fluid with large external forces
in a bounded domain $\Omega\subset{\mathbb R}^d$ ($d=2,3$), provided that
the Mach number is appropriately small. At the same time,
the low Mach number limit is rigorously verified. The basic idea in the proof is to split
the equations into two parts, one of which is similar to the steady incompressible
Navier-Stokes equations with large forces, while another part corresponds to the steady
compressible heat-conductive Navier-Stokes equations with small forces.
The existence is then established by dealing with these two parts separately, establishing
uniform in the Mach number a priori estimates and exploiting
the known results on the steady incompressible Navier-Stokes equations.
\end{abstract}
\vspace{3mm}


\vspace{3mm}



\renewcommand{\theequation}{\thesection.\arabic{equation}}
\setcounter{equation}{0}
\section{Introduction}

This paper is mainly concerned with the existence of strong
solutions to the steady Navier-Stokes equations of a compressible
heat-conductive fluid in a bounded domain $\Omega\subset{\mathbb R}^d$ ($d=2,3$) with
large external forces: 
\begin{equation}
\label{scns}
\left\{\begin{array}{llll}
\text{div}(\varrho\mathbf{u}) = 0, \\[1mm]
\varrho\mathbf{u}\cdot\nabla \mathbf{u}+\nabla p = \mbox{div}\, \mathbb{S}(\nabla \mathbf{u})+\varrho\f +\g,
\\[1mm]
c_{_V}\varrho {\bf u}\cdot\nabla\Theta +p\, \mbox{div}\mathbf{u}=\kappa\triangle\Theta +\Psi.
\end{array}\right.
\end{equation}
Here $\varrho$ denotes the density, $\mathbf{u}\in\R^d$ the
velocity, $\Theta$ the temperature, $p=R\varrho\Theta$ the pressure with
$R > 0$ being the gas constant, $c_{_V}>0$ is the heat capacity at constant volume;
$\mathbf{f}$ is the density of external body force and $\g$ is a given external force.
The stress tensor $\mathbb{S}$ and the dissipation function $\Psi$ are defined by
\begin{align}
\mathbb{S}=2\mu D({\bf u})+\lambda \text{div}\mathbf{u}\,\mathbb{I},\quad
 \Psi=2\mu D({\bf u}): D({\bf u})+\lambda({\rm div}{\bf u})^2\ge0,\nonumber
\end{align}
$D({\bf u})=(\nabla \mathbf{u} +\nabla\mathbf{u}^t)/2$ is the deformation tensor.
The viscosity coefficients $\mu ,\lambda$ satisfy $2\mu +d\lambda \geq 0$ and $\mu >0$,
$\kappa >0$ is the heat conductivity coefficient.
Moreover, the total mass is prescribed:
$$\int_{\Omega}\varrho dx = M > 0.$$  

We impose that the velocity $\bu$ satisfies no-slip boundary
condition and the temperature $\Theta$ keeps constant on the boundary of $\Omega$ , i.e.,
\begin{equation}\label{bc'}
\bu=0,\quad\;\; \Theta=\vartheta_0\ \ \ {\rm on}\ \partial\Omega.
\end{equation}

In the last decades, the steady compressible heat-conductive Navier-Stokes equations have been studied by many
mathematicians and there are a lot of results on the existence in the literature,
here we recall some of them for both small and large external forces which are related
to our study in this paper, and we refer to the monograph \cite{NS04} for more details.
When external forces are sufficiently small, Matsumura and Nishida in 1982/83 proved
the existence of a solution with potential forces near a rest state \cite{MN82,MN83}, while
Valli and Zajackowski \cite{Valli83,VZ86} used the existence of global non-stationary solutions
to get the existence of stationary solutions. Later, Valli \cite{Valli87} showed
the existence of stationary solutions in the case of general forces by using an idea of Padula \cite{Padula87} to
decompose the equations into two parts that are governed by the Stokes equations and
a transport equation, respectively.
Beir\~ao da Veiga \cite{Veiga87} obtained more general existence results in the $L^p$-setting by
decomposing the equations into three parts
that are governed by the Stokes equations, a transport equation and the Laplace equation, respectively.
Another decomposition was studied in the paper \cite{NP94}.
In 1989, Farwig \cite{Farwig89} showed the existence of solutions to the steady
compressible heat-conductive Navier-Stokes equations for small forces with slip boundary condition.

When external forces are of arbitrary size, the existence of strong
solutions was proved in \cite{NP91,MP92} for the case of potential
forces. When the equations of state and the viscosity coefficients
satisfy certain (growth) conditions, Novotn\'y and Pokorn\'y
\cite{NP11,NP11-2} showed that weak or strong solutions to the
steady compressible heat-conductive Navier-Stokes equations exist.
Unfortunately, their results exclude the case of ideal polytropic gases,
for which the existence of strong solutions, to our best knowledge, still remains open.

The aim of the present paper is to establish the existence of strong solutions
to the steady compressible heat-conductive Navier-Stokes system (\ref{scns})
without any smallness assumption on the external forces $\f$ and $\g$, when the Mach number is small.

We mention that in the isentropic flow case, the existence of weak solutions or strong solutions
for large external forces has been extensively investigated. Lions
\cite{Lions98} first proved the existence of weak solutions under
the assumption that the specific heat ratio $\gamma>1$ in two
dimensions and $\gamma>5/3$ in three dimensions. The restriction on
$\gamma$ actually comes from the integrability of the density
$\varrho$ in $L^p$, and in fact, the higher integrability of
$\varrho$ has, the smaller $\gamma$ can be allowed. In \cite{NS04}
Novotn\'y and Stra\v{s}kraba showed the existence of weak solutions
for any $\gamma>3/2$ if $\f$ is potential and $\g =0$. By deriving a
new weighted estimate of the pressure, Frehse, Goj and Steinhauer
\cite{FGS}, Plotnikov and Sokolowski \cite{PS05} established an
improved integrability for the density under the assumption of the
$L^1$-boundedness of $\varrho\bu^2$ which was not shown to hold
unfortunately.
Plotnikov and Sokolowski \cite{PS07} proved the existence of
renormalized solutions to the Dirichlet boundary value problem for the compressible
Navier-Stokes equations for all $\gamma>4/3$.
In 2008, B\v{r}ezina and Novotn\'y \cite{BN08} was able to prove the existence of weak
solution to the spatially periodic problem for any
$\gamma>(3+\sqrt{41})/8$ when $\f$ is potential and $\g =0$, or for any
$\gamma>(1+\sqrt{13})/3\thickapprox1.53$ when $\f ,\g\in L^\infty$,
without assuming the $L^1$-boundedness of $\varrho \bu^2$, by
combining the $L^\infty$-estimate of $\triangle^{-1}P$ with the
(usual) energy and density bounds. 
Then, in the framework of
\cite{PS05-1}, Frehse, Steinhauer and Weigant \cite{FSW,FSW09}
established the existence of weak solutions to the Dirichlet boundary value problem
for any $\gamma>4/3$ in three dimensions and to the spatially periodic or mixed
boundary value problem for $\gamma=1$ (isothermal flow) in two dimensions.
Recently, Jiang and Zhou \cite{JZ11}
proved the existence of weak solutions to the spatially periodic problem in $\mathbb{R}^3$ for any $\gamma>1$
by establishing a new weighted estimate. More recently, the existence for the slip and
Dirichlet boundary value problems for $\gamma>1$ was shown by Jessl\'e and Novotn\'y \cite{JN13}, and Plotnikov and Weigant
\cite{PW13}, respectively.
Furthermore, we emphasize that the existence of strong solutions was shown by Choe and Jin  \cite{CJ00}
when the Mach number is small, by exploiting the known results for the incompressible
steady Navier-Stokes equations.

Now, we rewrite (\ref{scns}) in the form of the Mach number.
After scaling and a straightforward calculation we obtain
the following dimensionless form of the steady full compressible Navier-Stokes equations:
\begin{equation}
\label{1'}
\left\{\begin{array}{llll}
\text{div}(\varrho\mathbf{u}) = 0, \\[1mm]
\varrho\mathbf{u}\cdot\nabla \mathbf{u}+\displaystyle{\frac{\nabla p}{\epsilon^2}}
= \text{div} \mathbb{S}(\nabla \mathbf{u})+\varrho\f +\g,\\[1mm]
\varrho {\bf u}\cdot\nabla\Theta +p\,{\rm div}{\bf u}=\kappa\triangle\Theta +\epsilon^2\Psi ,
\end{array}\right.
\end{equation}
where $\epsilon$ is the Mach number.

Since the total mass of the fluid is given, we impose the condition
$$\bar{\varrho}:=\frac{1}{|{\Omega}|}\int_{\Omega}\varrho (x)dx>0, $$
which can be renormalized to $\bar\varrho=1$ without loss of generality. Similarly,
we also assume that $\bar{\Theta}=1$, $R=c_{_V}=1$, $\vartheta_0=1$.

To show the existence, we take the transformation
\begin{equation}
\varrho =1+\epsilon\rho,\quad \Theta =1+\epsilon\theta   \label{j1} \end{equation}
to rewrite the system (\ref{1'}) in the form:
\begin{equation}
\label{1}
\left\{\begin{array}{llll}
{\rm div}\bu+\epsilon{\rm div}(\rho{\bf u})=0,\\[1mm]
(1+\epsilon\rho)(\bu\cdot\nabla\bu)+\dis{ \frac{(1+\epsilon\theta)\nabla \rho}{\epsilon}
+\frac{(1+\epsilon\rho)\nabla\theta}{\epsilon} }
=\text{div} \mathbb{S}(\nabla \mathbf{u})+(1+\epsilon\rho)\f +\g, \\[4mm]
\epsilon(1+\epsilon\rho)\bu\cdot\nabla\theta+{\rm div}\bu
+(\epsilon\rho+\epsilon\theta+\epsilon^2\rho\theta){\rm div}\bu
=\epsilon\kappa\triangle\theta+\epsilon^2\Psi,
\end{array}\right.
\end{equation}
with boundary conditions
\begin{equation}\label{bc}
\bu=0,\quad\;\; \theta=0\ \ \ {\rm on}\ \partial\Omega.
\end{equation}

The low Mach number limit for the corresponding evolutionary equations
was investigated extensively. Here, we only refer to the non-isentropic Navier-Stokes equations,
whose analysis is more difficult due to the complexity of the system structure.
Hagstrom and Lorenz \cite{HL02} used a similar transformation to \eqref{j1}, the standard energy arguments
and decay estimates for heat kernels, and constructed a special symmetrizer for linear
hyperbolic-parabolic systems with large hyperbolic part satisfying the interaction
condition to get the low Mach number limit of compressible Navier-Stokes equations in $\R^n$.
 Bresch et al \cite{BDGL02} analyzed the acoustic waves by a method of characteristic expansions and
gave a formal asymptotics as $\epsilon\rightarrow0$ in a periodic domain under the assumption that the viscous heating and thermal diffusion
are negligible. Concerning the full compressible Navier-Stokes equations, Alazard \cite{Alazard06} studied this singular
limit for local $H^s$ solutions ($s> 2+\frac{n}{2}$) in $\R^n$ for "ill-prepared" initial data by employing the technique of
pseudo-differential operators which does not apply to the cases with boundary due to the restriction of the
Fourier transform. As an improvement of \cite{BDGL02}, Feireisl and Novotn\'y \cite{FN07} considered the low Mach number limit
for the periodic "variational solutions" to the full Navier-Stokes-Fourier equations of certain radiative gases
for "ill-prepared" initial data. Related progresses for bounded domains with various boundary conditions
can be found in \cite{FMN08,FNP08,DF09}. On the other hand, when the thermal conductivity vanishes, the
incompressible limit of the non-isentropic Navier-Stokes equations in bounded domains has been studied recently
in \cite{JO11} where the strong solutions are shown to be bounded uniformly in a local time interval
for ``well-prepared'' initial data, thus implying the limit. However, to our best knowledge,
%
%
the low Mach number limit for the non-isentropic Navier-Stokes equations governing polytropic gases
in bounded domains is not yet proved so far, and the aim of the current paper is thus to show this limit.
%
%

Now, we state the main result of this paper.
\begin{theorem}
\label{thm}
Let $\f,\g\in H^2(\Omega )$. Then there is an $\epsilon_0$ depending on $\|(\f,\g)\|_{H^2}$ and $\Omega$,
such that for any $\epsilon\in (0,\epsilon_0)$,
 there exists a solution
 $(\rho^\epsilon,{\bf u}^\epsilon,\theta^\epsilon)\in\bar H^2\times (H^3\cap H_0^1)\times (H^3\cap H_0^1)$
 to the boundary value problem \eqref{1}, \eqref{bc}, satisfying
$$\lim_{\epsilon\rightarrow0}\inf_{U,P\in \mathbf{L}}\|\bu^\epsilon-U\|_3 +\|\rho^\epsilon\|_2
+\|\theta^\epsilon\|_3+\Big\|\frac{\rho^\epsilon+\theta^\epsilon}{\epsilon}-P\Big\|_2=0,$$
where $(U,P)\in\mathbf{L}:=\{(U,P)\in (H^4\cap H_0^1)\times\bar H^3\;|\, (U,P)$ is a
solution of the incompressible steady Navier-Stokes equations \eqref{ISNS} with external force
$\f +\g$\}, i.e.,
\begin{equation}\label{ISNS}
\left\{\begin{array}{llll}
U\cdot\nabla U-\mu\triangle U+\nabla P=\f +\g ,  \\[1mm]
\mathrm{div}\, U=0,\\[1mm]
U=0\ \ {\rm on}\ \partial\Omega,\ \ \ \ \dis{\int_\Omega Pdx=0.}
\end{array}\right.
\end{equation}
\end{theorem}
\begin{remark}
If considering the existence of spatially periodic solutions to \eqref{1} in a periodic domain,
we can also obtain a existence result similar to Theorem \ref{thm}.
\end{remark}
\begin{remark}
In general, we could not get the uniqueness of strong solutions
to the boundary value problem \eqref{scns} due to lack of the
uniqueness of strong solutions to the corresponding incompressible
Navier-Stokes equations \eqref{ISNS}. As for the results of the
incompressible Navier-Stokes equations, one can refer to Galdi's book \cite[Chapter IX]{Galdi94}.
\end{remark}

The system (\ref{1}) is complicated mixed-type nonlinear equations
containing such structures as elliptic and hyperbolic systems, for
which the usual approach of the fixed point arguments used to prove
the existence of classical solutions requires the smallness of data. To show Theorem \ref{thm}
we split the system (\ref{1}) into two parts, one of which is similar to the
steady incompressible Navier-Stokes equations with large force $\f+\g$,
while another part corresponds to the steady compressible
heat-conductive Navier-Stokes equations with small force $\epsilon\f$,
provided the Mach number $\epsilon$ is small. Then, as noted in \cite{CJ00}, we modify and elaborately combine
the arguments in \cite{Galdi94} where the existence of strong solutions
to the incompressible Navier-Stokes equations for large forces was presented, and
in \cite{Farwig89} where strong solutions of the compressible viscous
heat-conductive equations with small forces were dealt with, to establish Theorem \ref{thm}.

Compared with the isentropic case studied in \cite{CJ00}, due to
presence of the energy equation, the main difficulties here lie in
obtaining the existence of weak solutions to the linearized system \eqref{CL},
dealing with the coupling terms between the velocity,
density and temperature, and deriving the uniform-in-$\epsilon$ estimates in a bounded domain,
for example, how to control the energy norm
$\|\bu-U\|_3+\|\eta\|_2+\|\theta\|_3$ uniformly in $\epsilon$ under
the no-slip boundary condition. To circumvent such difficulties, we
take the transform of $\varrho =1+\epsilon\rho$,
$\Theta=1+\epsilon\theta$ for the system (\ref{1'}) (see also \cite{HL02}),
instead of the transform ($\varrho=1+\epsilon^2\rho$, $\Theta=1+\epsilon^2\theta$)
used in \cite{CJ00}, carefully construct an approximate linear problem by using a cut-off function
and employ a mollifier technique to get the existence and uniqueness of solutions to \eqref{CL},
and utilize the lower order terms to control the higher order terms.
We should remark that in \cite{CJ00} the linearized problems are decoupled, while,
to obtain the uniform a priori estimates in this paper, the linearized problems are strongly coupled that
gives rise to more difficulties than the isentropic case in \cite{CJ00}.

Let us briefly explain the main steps of our proof. First, we construct the approximate linear
problem, derive the uniform estimates and employ a compactness argument to get
the existence of a weak solution to the system \eqref{CL} in Section 2.2.
Second, we exploit the property of the momentum equations and the regularity
of the Stokes problem to establish the estimates of $|\eta+\theta|$,
which, combined with an estimate for $\theta$, implies the boundedness for
$\eta$. Due to presence of boundary here, some difficulties involved
with controlling the boundary terms arise. To overcome such
difficulties, the crucial step is to get a $H^2$-bound of
${\rm div}\bu$ near the boundary, for which we shall adopt the local
isothermal coordinates used in \cite{Valli83, VZ86}. This
strategy has also been used in \cite{JO11} to study the
low Mach number limit of the compressible Navier-Stokes equations with non-slip
boundary condition. Then, summing up all the estimates for
$(\bv,\theta)$ and $\eta$, we can establish the desired a priori
uniform in $\epsilon$ estimates in view of the smallness of
$\epsilon$ (see Section 3). Finally, we apply the Tikhonov fixed
point theorem to obtain the existence of a strong solution.
Moreover, with the help of the uniform a priori estimates, one can
take the limit to show the incompressible limit.
We point out here that due to the splitting, we have to impose that the energy equation should
not possess an external heat source.

The rest of this paper is organized as follows. In the next section, we
prove the existence of weak solutions and regularity to the linearized incompressible and
compressible problems.
Section 3 is devoted to establishing the existence for the nonlinear
problem. Finally, the incompressible limit of solutions to the steady compressible heat-conductive
Navier-Stokes equations is presented in Section 4.

{\it Notations:} We denote by $L^2$ the Lebesgue space $L^2(\Omega)$ with
norm $\|\cdot\|_0$, by $H^m$ the Sobolev spaces $H^m(\Omega )$ with norm $\|\cdot\|_m$. Define the spaces
$$\bar H^m=\Big\{\rho\in H^m\; |\; \int_\Omega\rho (x)dx=0\Big\},\quad
H_{0,\sigma}^1=\Big\{\bu\in H_0^1\;|\; {\rm div}\bu=0\Big\},\quad H_{0,\sigma}^m:=H^m\cap H_{0,\sigma}^1.$$
We denote by $H^{-1}$ the dual space of $H_0^1$
with the dual product $\langle\cdot,\cdot\rangle$ and the norm
$\|\cdot\|_{-1}=\sup_{\|h\|_1=1}|\langle\cdot,h\rangle|.$ We shall use the abbreviation:
$$\int \,\cdot\; dx :=\int_\Omega \,\cdot\; dx. $$

\section{Existence of solutions to the linearized problem}
\newcounter{linear}
\renewcommand{\theequation}{\thesection.\thelinear}

We first split the system \eqref{1} into two parts, so that one part
looks like the incompressible Navier-Stokes equations, while the other part
behaves like the compressible Navier-Stokes equations. More precisely,
let $(U,P)$ and $(\bv,\eta)$ be the solutions to the
following systems, respectively: \stepcounter{linear}
\begin{equation}
\label{i1}
\left\{\begin{array}{llll}
U\cdot\nabla U+\bv\cdot\nabla U-\mu\triangle U+\nabla P=\f +\g,
\\[1mm]
{\rm div}\, U=0,\\[1mm]
U=0\ \ \ \  {\rm on}\ \ \partial\Omega\ \ \ {\rm and}\quad \dis{\int_\Omega Pdx=0};
\end{array}\right.
\end{equation}
and
\stepcounter{linear}
\begin{equation}
\label{c1}
\left\{\begin{array}{llll}
U\cdot\nabla \eta +\dis{\frac{{\rm div}\bv}{\epsilon} }
=-\bv\cdot\nabla\eta-\eta\,{\rm div}\bv-\epsilon\,{\rm div}(P(U+\bv)),\\[1mm]
 U\cdot\nabla \bv-\mu\triangle \bv-(\mu+\lambda)\nabla{\rm div}\bv
+ \dis{ \frac{\nabla \eta+\nabla\theta}{\epsilon} }
=\epsilon F-\bv\cdot\nabla \bv-\theta\nabla\eta-\eta\nabla\theta,\\[1mm]
U\cdot\nabla\theta-\kappa\triangle\theta +\dis{ \frac{{\rm div}\bv}{\epsilon} }
=\epsilon G-\bv\cdot\nabla\theta-\eta{\rm div}\bv-\theta{\rm div}\bv,\\[1mm]
\bv =0,\;\;\;\theta=0\;\; {\rm on}\;\;\partial\Omega\quad\; {\rm and}\quad \dis{\int_\Omega \eta dx =0},
\end{array}\right.
\end{equation}
where the new force $F$ and heat source $G$ are defined by
\begin{eqnarray*} &&
F=(\epsilon P+\eta)\f -(\epsilon P+\eta)(U+\bv)\cdot\nabla(U+\bv)-\theta\nabla P- P\nabla\theta , \\
&&
G=\Psi-(\epsilon P+\eta)(U+\bv)\cdot\nabla\theta-(\epsilon P+\eta)\theta{\rm div}\bv-P{\rm div}\bv .
\end{eqnarray*}

It is clear to observe that $\mathbf{u}:=U+\bv$, $\rho :=\epsilon P+\eta$ and $\theta$
are a solution to \eqref{1}. Thus, we can obtain a solution of the system \eqref{1}
if we can solve the systems \eqref{i1} and \eqref{c1}. First, we will give the existence
of weak solutions to the linearized incompressible problem \eqref{i1} and derive
a priori estimates of higher order derivatives of the unknowns $(U,P)$. Then, we shall show the existence
of weak solutions to the linearized compressible problem \eqref{c1} and establish uniform
estimates of higher order derivatives of the unknowns $(\eta, \bv, \theta)$.

In what follows, we assume that meas$(\Omega) =1$ without loss of generality.

\subsection{Linearized incompressible equations}

Let $\tilde U$ and $\tilde \bv$ be given functions satisfying
$\tilde U\in  H^4\cap H_{0,\sigma}^1$ and $\tilde \bv\in  H^3\cap H_0^1$. At first, we consider
the linearized equations to \eqref{i1} for given $\tilde U$ and $\tilde \bv$ as follows.
\stepcounter{linear}
\begin{equation}
\label{i1l}
\left\{\begin{array}{llll}
(\tilde U+\tilde \bv)\cdot\nabla U-\mu\triangle U+\nabla P=\h ,  \\[1mm]
{\rm div}U=0,\\[1mm]
U=0\;\; {\rm on}\;\;\partial\Omega\quad\; {\rm and}\quad\;\dis{\int_\Omega Pdx=0}.
\end{array}\right.
\end{equation}
where $\h=\f+\g.$

The problem \eqref{i1l} is a Stokes problem which is solvable for arbitrarily large forces.
In fact, \eqref{i1l} can be solved by using the Lax-Milgram theorem for small $\tilde \bv$,
and we can obtain the following existence result, the proof of which can be found, for example,
in \cite{CJ00}, and is therefore omitted here.
\begin{lemma}\label{lmI1}
Let $\h\in H^{-1}$, $\tilde U\in  H_{0,\sigma}^1$ and $\tilde \bv\in  H^3\cap H_0^1$. There exists a constant $a_0$ depending
only on $\mu$ and $\Omega$, such that if $\|\tbv\|_3< a_0$, then there exists a weak solution $(U,P)\in H_0^1\times \bar H^0$
 of \eqref{i1l}, satisfying
\stepcounter{linear}
\begin{eqnarray}   &&  \label{lml1-1}
\|U\|_{1}\le C_0\|\h\|_{-1}, \\[1mm]
 \stepcounter{linear}
&& \label{lml1-2} \|P\|_{0}\le C_1\|\h\|_{-1}(1+\|\h\|_{-1}),
\end{eqnarray}
where $C_0$ and $C_1$ are positive constants which depend only on $\Omega,\ \mu\ {\rm and}\ a_0$.
\end{lemma}

As for the regularity of solutions, we consider the Stokes equations:
\begin{align*}
-\mu\triangle U+\nabla P&=\h-(\tilde U+\tilde \bv)\cdot\nabla U,\\
{\rm div}U&=0.
\end{align*}
Then we can derive the following estimates by employing bootstrap arguments similar to those
in \cite{CJ00}.
\begin{lemma}\label{lmI2}
Let $\h\in H^m,\ \tilde U\in H^{m+1}\cap H_{0,\sigma}^1,\ m=0,1,2,$ and $\tbv$ be the same as in Lemma \ref{lmI1}.
There are positive constants $C_2,\ C_3$ and $C_4$, depending only on $\Omega,\ \mu\ {\rm and}\ a_0$,
such that if $\tilde U\in H_0^1$ satisfies the inequality \eqref{lml1-1}, then
\stepcounter{linear}
\begin{equation}\label{lml2-1}
\|U\|_{2}+\|\nabla P\|_0\le C_2\|\h\|_{0}(\|\h\|_0+1)^4 .
\end{equation}
If $\tilde U\in H^2\cap H_0^1$ satisfies \eqref{lml2-1}, then
\stepcounter{linear}
\begin{equation}\label{lml2-2}
\|U\|_{3}+\|\nabla P\|_1\le C_3\|\h\|_{1}(\|\h\|_1+1)^8,
\end{equation}
and if $\tilde U\in H^3\cap H_0^1$ satisfies \eqref{lml2-2}, then
\stepcounter{linear}
\begin{equation}\label{lml2-3}
\|U\|_{4}+\|\nabla P\|_2\le C_4\|\h\|_{2}(\|\h\|_2+1)^{12}.
\end{equation}
\end{lemma}

Let $\f ,\g\in H^2(\Omega)$, then it is obvious that $\h\in H^2(\Omega)$. We define a space $K_0$ by
\stepcounter{linear}
\begin{eqnarray}
K_0: =&&\Big\{U\in H_{0,\sigma}^4(\Omega):\; \|U\|_1\le C_1\|\h\|_1,\ \|U\|_{2}\le C_2\|\h\|_{0}(\|\h\|_0+1)^4,
\nonumber \\
&& \qquad \|U\|_{3}\le C_3\|\h\|_{1}(\|\h\|_1+1)^8,\ \|U\|_{4}\le C_4\|\h\|_{2}(\|\h\|_2+1)^{12}\Big\}.  \label{k0}
\end{eqnarray}

Thus, by Lemma \ref{lmI2} we see that the solution $U$ of the system \eqref{i1l}
also lies in $K_0$ for any given $\tilde U\in K_0$, since the constants $C_1,\cdots, C_4$ do not depend on  $\tilde U$.

\subsection{Linearized compressible equations}

Let $(\tilde U,\tbv,\tilde\theta)\in (H^4\cap H_{0,\sigma}^1)\times
(H^3\cap H_0^1)\times (H^3\cap H_0^1)$ be given functions.
Next, we give the existence of weak solutions and derive some a
priori estimates for solutions to the linearized equations of the
system (\ref{c1}). For simplicity, we only consider the
three-dimensional case. As aforementioned, we shall apply the
Tikhonov fixed point theorem to show the existence of strong
solutions to (\ref{1}), (\ref{bc}). To this end, for given
$(\tilde \bv,\tilde\theta)\in(H^3\cap H_0^1)\times(H^3\cap H_0^1)$, let $(\eta,\bv,\theta)$
be the unique solution of the following linearized system of (\ref{c1}) the existence of
which will be shown below:
\stepcounter{linear}
\begin{equation}
\label{CL}
\left\{\begin{array}{llll}
U\cdot\nabla \eta + \dis{\frac{{\rm div}\bv}{\epsilon} }+\tilde \bv\cdot\nabla\eta+\eta{\rm div}\tbv
=-\epsilon{\rm div}(P(U+\tilde \bv)),\\[1mm]
U\cdot\nabla \bv-\mu\triangle \bv-\zeta\nabla{\rm div}\bv+ \dis{\frac{\nabla \eta+\nabla\theta}{\epsilon} }+\tilde \theta\nabla\eta +\eta\nabla\tilde \theta
=\epsilon \tilde F-\tilde \bv\cdot\nabla \tilde \bv,\\[1mm]
U\cdot\nabla\theta-\kappa\triangle\theta+ \dis{\frac{{\rm div}\bv}{\epsilon} }+\eta{\rm div}\tbv
=\epsilon \tilde G-\tilde \bv\cdot\nabla\tilde\theta -\tilde \theta{\rm div}\tbv ,\\[1mm]
\bv=0,\ \ \ \ \theta=0\ \  {\rm on}\ \ \partial\Omega\ \ \ {\rm and}\quad \dis{\int_\Omega \eta=0},
\end{array}\right.
\end{equation}
where $(U,P)$ be the solution of \eqref{i1l} established in Section 2.1 and the new force $\tilde F$ and heat source $\tilde G$ are defined by
\begin{eqnarray*} &&
\tilde F=(\epsilon P+\eta)\f -(\epsilon P+\eta)(U+\tilde \bv)\cdot\nabla(U+\tilde \bv)
-\tilde \theta\nabla P- P\nabla\tilde \theta, \\[1mm]
&& \tilde G=\tilde \Psi-(\epsilon P+\eta)(U+\tilde \bv)\cdot\nabla\tilde \theta
-(\epsilon P+\eta)\tilde\theta{\rm div}\tbv-P{\rm div}\tbv,
\end{eqnarray*}
with $\tilde \Psi=2\mu D(U+\tilde \bv): D(U+\tilde \bv)+\lambda({\rm div}(U+\tilde \bv))^2$
and $\zeta=\mu+\lambda$.

Thus, for given $\tilde U, \tilde \bv$ and $\tilde\theta$, we can construct a map $N$:
$$ N(\tilde U, \tilde \bv, \tilde\theta):=(U, \bv, \theta).$$
And, we have to show that $N$ maps some space into itself and is weak continuous to get a fixed point of the mapping $N$.

In order to obtain the existence of weak solutions to \eqref{CL}, we set
\begin{eqnarray*} &&
\tilde F'=\epsilon P\f -\epsilon P(U+\tilde \bv)\cdot\nabla(U+\tilde \bv)
-\tilde \theta\nabla P- P\nabla\tilde \theta, \\[1mm]
&& \tilde G'=\tilde \Psi-\epsilon P(U+\tilde \bv)\cdot\nabla\tilde \theta
-\epsilon P\tilde\theta{\rm div}\tbv-P\,{\rm div}\tbv.
\end{eqnarray*}
So, $\tilde F=\tilde F'+\eta \f+\eta(U+\tilde\bv)\cdot\nabla(U+\tbv),\ \ \tilde G=\tilde G'-\eta(U+\tbv)\cdot\nabla\tilde\theta-\eta\tilde\theta{\rm div}\tbv$.

\subsubsection{Existence of weak solutions}
\begin{lemma}\label{lmce}{
Let $\tilde F',\tilde G'\in H^{-1}$}, $\f\in H^2$, and $(U,P)$ be a solution of \eqref{i1l}
established in Lemma \ref{lmI1}. If $\|\tbv\|_3+\|\tilde\theta\|_3$ is
sufficiently small, then there exists a unique weak solution
$(\eta,\bv,\theta)\in\bar{H}^0\times{H}_0^1\times{H}_0^1$ to the problem \eqref{CL}.
\end{lemma}
\begin{remark}
We remark here that a weak solution of (\ref{CL}) is defined through testing the system (\ref{CL})$_1$--(\ref{CL})$_3$ by
$\phi\in H^1$, $\mathbf{V}\in H^1_0$ and $\Theta\in H^1_0$, respectively.

\end{remark}
\begin{remark}
In \cite{PRS09} the authors used the Bergman projection to reduce the linearized problem to a boundary value problem
for the transport operator equation and then proved the existence of strong solutions and the incompressible limit as
$\omega={\rm Re}/(\gamma {\rm Ma}^2)$ goes to $\infty$. Here we shall employ the techniques of approximate systems,
the Leray-Schauder fixed theorem and compactness arguments and mollifiers to show the existence and uniqueness
of weak solutions directly.
\end{remark}
\begin{proof}
The proof is broken up into three steps.

\emph{Step 1: Construction of strong solutions to the approximate problems.}

We first suppose that $\tilde F',\tilde G'\in L^{2}$ and the case $\tilde F',\tilde G'\in H^{-1}$ will then be dealt with
by a density argument later.  Define the cut-off function
$$1_\alpha(\eta)=\left\{
   \begin{array}{ll}
    1/\alpha, & \hbox{ if }\eta>1/\alpha; \\
      \eta, & \hbox{ if }\eta\leq 1/\alpha.
    \end{array}
  \right.
$$
Now, we consider the following approximate system to (\ref{CL}) for arbitrary but fixed positive numbers $\beta$, $\alpha$:
\stepcounter{linear}
\begin{equation}
\label{CLa}
\left\{\begin{array}{llll}
-\beta\Delta \eta = - \mathrm{div}(U\eta)- \dis{\frac{{\rm div}\bv}{\epsilon} }-\mathrm{div}[1_\alpha(\eta)\tbv]
-\epsilon{\rm div}(P(U+\tbv))=:\mathrm{div}R_1(\eta,\mathbf{v}),\\[3mm]
-\mu\triangle \bv -\zeta\nabla{\rm div}\bv=-
 \dis{\frac{\nabla ( \eta + \theta)}{\epsilon} }-U\cdot\nabla \bv-\nabla[\tilde \theta 1_\alpha(\eta)]\\[2mm]
\quad+\epsilon [\tilde F^{'}+1_\alpha(\eta)\mathbf{f}+1_\alpha (\eta)(U+\tbv)
\cdot\nabla(U+\tbv)]-\tilde \bv\cdot\nabla \tilde \bv =:R_2(\eta,\mathbf{v},\theta), \\[3mm]
-\kappa\triangle\theta= -U\cdot\nabla \theta-
\dis{\frac{{\rm div}\bv}{\epsilon} }- 1_\alpha(\eta){\rm div}\tbv  \\[2mm]
\quad
+\epsilon\{\tilde G^{'}- 1_\alpha(\eta)[(U+\tbv)\cdot\nabla\tilde\theta+\tilde\theta{\rm div}\tbv]\}
-\tilde \bv\cdot\nabla\tilde\theta -\tilde \theta{\rm div}\tbv:= R_3(\eta,\mathbf{v},\theta),\\[3mm]
\displaystyle\bv =0,\ \ \ \ \theta =0,\ \  \
\frac{\partial \eta}{\partial \mathbf{n}}=0
\quad {\rm on}\ \ \partial\Omega,\ \ \quad{\rm and }\ \ \dis{\int_\Omega\eta dx =0},
\end{array}\right.
\end{equation}
where $\mathbf{n}$ denotes the outer normal vector of the boundary $\partial \Omega$. We point out that the introduction of
the cut-off function makes it possible to establish an important estimate \eqref{fsdf14f} below.

Next, we show the strong solvability of \eqref{CLa} by employing the Leray-Schauder fixed point theorem.
To this end, we first study the following boundary problem:
 \stepcounter{linear}\begin{equation}
 \label{Appdasdf01}
\left\{\begin{array}{llll}
-\beta\Delta \eta
={\rm div}R_1(t\chi,t\mathbf{u}),\\[1mm]
-\mu\triangle \bv -\zeta\nabla{\rm div}\bv
=R_2(t\chi,t\mathbf{u},t\vartheta),\\[1mm]
-\kappa\triangle\theta=R_3(t\chi,t\mathbf{u},t\vartheta),\\[1mm]
\displaystyle\bv =0,\ \ \ \ \theta =0,\ \  \
\frac{\partial \eta}{\partial \mathbf{n}}=0
\quad
{\rm on}\ \ \partial\Omega,\ \ \quad{\rm and }\ \ \dis{\int_\Omega \eta  =0}
\end{array}\right.
\end{equation}
 for given $t\in [0,1]$ and $({\chi},\mathbf{u}, \vartheta)
\in\bar{H}^1\times H^1_0\times H^1_0$.

 Recalling that $(\tilde U,\tbv,\tilde\theta)\in (H^4\cap H_{0,\sigma}^1)\times
(H^3\cap H_0^1)\times (H^3\cap H_0^1)$
 and $1_\alpha(\eta)\in H^1$,
by  the elliptic theory (see \cite[Lemma 4.32 and Lemma 4.27]{NS04}),
we know that there exits a unique strong solution $(\eta, \mathbf{v},\theta)\in H^2$ satisfying \eqref{Appdasdf01}
and
 \stepcounter{linear}\begin{equation}
 \label{Apdsdf01asfa}
\|( \eta,\theta,\mathbf{v})\|_2\leq C ( \|\mathrm{div}R_1(t\chi,t\mathbf{u})\|_0 +
\|(R_1(t\chi,t\mathbf{u}),R_2(t\chi,t\mathbf{u},t\vartheta),R_3(t\chi,t\mathbf{u},t\vartheta))\|_0).
\end{equation}

Then, we derive estimates for strong solutions $(\eta,\mathbf{v},\theta)$ of the problem:
 \stepcounter{linear} \begin{equation}    \label{Appro2}
\left\{\begin{array}{llll}
-\beta\Delta \eta
={\rm div}R_1(t\eta,t\mathbf{v}),\\[1mm]
-\mu\triangle \bv -\zeta\nabla{\rm div}\bv
=R_2(t\eta,t\mathbf{v},t\theta),\\[1mm]
-\kappa\triangle\theta = R_3(t\eta,t\mathbf{v},t\theta),\\[1mm]
\displaystyle\bv =0,\quad \theta =0,\quad \frac{\partial \eta}{\partial \mathbf{n}}=0
\;\;{\rm on }\; \partial\Omega, \quad{\rm and }\; \dis{\int_\Omega \eta  =0},\quad\; t\in [0,1].
\end{array}\right.
\end{equation}

Using $L^2$ energy estimates, we can easily obtain the following identities on $(\eta,\mathbf{v},\theta)$:
\stepcounter{linear}
\begin{eqnarray}
&& \int\mu|\nabla\bv|^2+\zeta({\rm div}\bv)^2dx
=\frac{1}{\epsilon}\int t(\eta+\theta){\rm div}\bv dx+\int  1_\alpha(t\eta)\tilde\theta{\rm div}\bv dx \nonumber\\
&&\qquad +\int \{\epsilon [\tilde F^{'}+1_\alpha(t\eta)\mathbf{f}+1_\alpha (t\eta)(U+\tbv)
\cdot\nabla(U+\tbv)]-\tilde \bv\cdot\nabla \tilde \bv\} \bv dx,\label{djy1}   \end{eqnarray}
\stepcounter{linear}
\begin{equation}
\beta\int|\nabla \eta|^2dx+\frac{1}{\epsilon}\int\eta{\rm div}\bv dx
=-\epsilon\int{\rm div}(P(U+\tbv))\eta dx-\frac12\int \eta1_\alpha(t\eta){\rm div}\tbv dx,\label{djy2} \end{equation}
\stepcounter{linear}
\begin{eqnarray}
&& \int\kappa|\nabla\theta|^2dx+\frac{1}{\epsilon}\int\theta{\rm div}\bv dx
=\epsilon\int\tilde{G^{'}}\theta dx-\epsilon\int(1_\alpha( t\eta)(U+\tbv)\cdot\nabla\tilde\theta
+1_\alpha(t\eta)\tilde\theta{\rm div}\tbv)\theta dx\nonumber\\[1mm]
&& \qquad -\int(\tbv\cdot\nabla\tilde\theta+\tilde\theta{\rm div}\tbv)\theta dx-\int1_\alpha(t\eta)\theta{\rm div}\tbv dx.\label{djy3}
\end{eqnarray}
Keeping in mind that $t\leq 1$, $\|1_\alpha(\eta)\|_0^2\leq |\Omega|/\alpha$ and $\|\eta\|_0\leq C \|\nabla \eta\|_0$,
summing up the identities \eqref{djy1}--\eqref{djy3}, we deduce the following estimate:
 \stepcounter{linear}\begin{equation}
 \label{fsdf14f}\|(\eta ,\mathbf{v} ,\theta )\|_1 <K \equiv K(1+\|(\tilde{F}',\tilde{G}')\|_0 +\|\mathbf{f}\|_1
 +\|(U,\tilde{\mathbf{v}},\tilde{\theta},P)\|_2^2),\end{equation}
where the constant $K$ may depend on other known quantities, but is independent of $t$.

Now for given $({\chi},\mathbf{u}, \vartheta)\in\bar{H}^1\times H^1_0\times H^1_0$, we define a mapping $T_t$ by
 $$ T_t(\chi,\mathbf{u},\vartheta):=(\eta_t,\mathbf{v}_t,\theta_t)\;\mbox{ that is the strong solution of \eqref{Appdasdf01}}.$$
 Then the solution operator $T_t$ enjoys the following properties (in what follows, $C$ denotes a general constant independent of $t$):
\begin{itemize}
  \item  By virtue of \eqref{fsdf14f},
\stepcounter{linear}
\begin{equation}
   0\not\in (I-T_t)(\partial B_{K}),
\end{equation}
where
$$B_K=\Big\{(\eta ,\mathbf{v},\theta )\in H^1\times H_0^1\times H^1_0~\bigg|~\|(\eta ,\mathbf{v},\theta)\|_1 <K,
\quad \int_\Omega \eta=0 \Big\}.$$
  \item Let $B$ be a ball of center $0$ in $\bar{H}^1\times H^1_0\times H^1_0$,
then for any $t\in [0,1]$, $T_tB$ is a precompact set of $\bar{H}^1\times H^1_0\times H^1_0$, since one has by \eqref{Apdsdf01asfa} that
$$\|T_t(\chi,\mathbf{u},\vartheta)\|_2\leq C(\|(\chi,\mathbf{u},\vartheta)\|_1 +1)
(1+\|(\tilde{F}',\tilde{G}')\|_0 +\|\mathbf{f}\|_1 +\|(U,\tilde{\mathbf{v}},\tilde{\theta},P)\|_2^2).$$
\item For any $t,s\in [0,1]$, one has by the elliptic regularity theory (see \cite[Lemma 4.32 and Lemma 4.27]{NS04}) that
\stepcounter{linear}\begin{equation}
\begin{aligned}
&\|T_t(\chi,\mathbf{u},\vartheta)-T_s(\chi,\mathbf{u},\vartheta)\|_1\leq C(\|R_1(t\chi,t\mathbf{u})-R_1(s\chi,s\mathbf{u})\|_0 \\
& \qquad +\|(R_2(t\chi,t\mathbf{u},t\vartheta),R_3(t\chi,t\mathbf{u},t\vartheta))
-(R_2(s\chi,s\mathbf{u},s\vartheta),R_3(s\chi,s\mathbf{u},s\vartheta))\|_{-1}  \\
 &\leq C|t-s|\, \|(\xi,\mathbf{u},\vartheta)\|_1\big( 1+\|(U,\tilde{\mathbf{v}}, \tilde{\theta},P)\|_2^2\big).
\end{aligned}
\end{equation}
\end{itemize}

Consequently, we can applying the Leray-Schauder fixed point theorem \cite[Section 1.4.11.7]{NS04} to \eqref{Appdasdf01} to
get a strong solution $(\eta, \mathbf{v},\theta)$ of the approximate system \eqref{CLa}. In what follow we write such
strong solution by $(\eta_{\beta}^\alpha, \mathbf{v}_{\beta}^\alpha,\theta_{\beta}^\alpha)$ to indicate the dependence upon
$\alpha$ and $\beta>0$.
\vspace{2mm}

\emph{Step 2: Uniform-in-$(\alpha,\beta)$ estimates and the limit process}

Recalling that the strong solutions $(\eta_{\beta}^\alpha, \mathbf{v}_{\beta}^\alpha,\theta_{\beta}^\alpha)$
constructed above satisfy \eqref{djy1}--\eqref{djy3} for $t=1$, and
the fact $\left|1_{\alpha}(\eta_{\beta}^\alpha)\right|\leq |\eta_{\beta}^\alpha|$, we find that
\stepcounter{linear}
\begin{eqnarray}
&&\|\bv_{\beta}^\alpha\|_1^2+\|\theta_{\beta}^\alpha\|_1^2+\beta\|\nabla \eta_{\beta}^\alpha\|^2_0
\le C \|\eta_{\beta}^\alpha\|_0^2 \Big[\|\tilde\theta\|_2 +\epsilon(\|\f\|_1 +\|U+\tbv\|_2^4)  \nonumber \\
&& \;\; +\epsilon\|\tbv\|_3 +\epsilon\|\tilde\theta\|_2^2(\|U+\tbv\|_{H^1}^2+\|\tbv\|_2^2)
 +\|\tbv\|_2^2\Big] +C\epsilon\|\eta_{\beta}^\alpha\|_0\|P\|_2\|U+\tbv\|_2  \nonumber   \\
&& \;\; +C\epsilon(\|\tilde{F^{'}}\|_{-1}^2 +\|\tilde{G^{'}}\|_{-1}^2)
+C\|\tbv\|_1^4 +C\|\tbv\|_1^2 \|\tilde\theta\|_1^2,   \label{djy4}
\end{eqnarray}
provided $\|\tbv\|_3+\|\tilde\theta\|_3$ and $\epsilon$ are sufficiently small. We remark that throughout the proof
of Lemma 2.3, the smallness of $\|\tbv\|_3+\|\tilde\theta\|_3$ and $\epsilon$ are independent of $\alpha$  and $\beta$.

We proceed to bound $\|\eta_{\beta}^\alpha\|_0$. According to the content on page 123 in \cite{PRS09},
for a function $\xi\in L^2(\Omega)$ with $\int\xi dx=0$,
one can choose a vector field $q\in H_0^1(\Omega)$, such that
$${\rm div}\, q=\xi,\quad\; \|q\|_1\le C(\Omega)\|\xi\|_0.$$
Multiplying $\eqref{CL}_2$ by this $q$ and integrating the resulting equation over $\Omega$, we obtain
\begin{eqnarray}
\frac{1}{\epsilon}\int(&&\eta_{\beta}^\alpha+\theta_{\beta}^\alpha){\rm div}q dx =-\int(\eta_{\beta}^\alpha\tilde\theta){\rm div}q dx-\epsilon\int[\tilde{F^{'}}+\eta_{\beta}^\alpha f+\eta_{\beta}^\alpha(U+\tbv)\cdot\nabla(U+\tbv)]\cdot q dx\nonumber\\[1mm]
&&+\int\tbv\cdot\nabla\tbv\cdot q dx+\int(U\cdot\nabla\bv)\cdot q dx
+\int\mu\nabla\bv\cdot\nabla q+\lambda{\rm div}\bv\cdot{\rm div} q dx ,\nonumber
\end{eqnarray}
which, by taking $\xi=\eta_{\beta}^\alpha+\theta_{\beta}^\alpha
-\frac{1}{|\Omega|}\int\theta_{\beta}^\alpha dx$ and applying Young's inequality, gives
\stepcounter{linear}
\begin{eqnarray}
\Big\|\eta_{\beta}^\alpha+\theta_{\beta}^\alpha - \frac{1}{|\Omega|}\int\theta_{\beta}^\alpha dx\Big\|_0^2
\le &&C\epsilon(\|\bv_{\beta}^\alpha\|_{H_0^1}^2\|U\|_1^2 + \|\eta_{\beta}^\alpha\|_0^2\|\tilde\theta\|_2^2
+\|\bv_{\beta}^\alpha\|_2^2 +\|\tbv\|_2^4)   \nonumber \\[1mm]
&& +C\epsilon^2[\|\tilde F\|_{-1}^2 +\|\eta_{\beta}^\alpha\|_0^2(\|\f\|_1^2 +\|U+\tbv\|_2^4)]. \label{djy5}
\end{eqnarray}
It is easy to see that
\stepcounter{linear}
\begin{eqnarray}
\Big\|\frac{1}{|\Omega|}\int\theta_{\beta}^\alpha dx\Big\|_0^2=\frac{1}{|\Omega|}
\Big(\int\theta_{\beta}^\alpha dx\Big)^2\le\|\theta_{\beta}^\alpha\|_0^2\label{djy6}
\end{eqnarray}
Hence, from \eqref{djy4}--\eqref{djy6} and Poincar\'e's inequality we get
\stepcounter{linear}
\begin{eqnarray}
&& \|\eta_{\beta}^\alpha\|_0^2\leq 2\|\eta_{\beta}^\alpha+\theta_{\beta}^\alpha\|_0^2
+2\|\theta_{\beta}^\alpha\|_0^2   \nonumber  \\
&& \le \Big\|\eta_{\beta}^\alpha+\theta_{\beta}^\alpha - \frac{1}{|\Omega|}\int\theta_{\beta}^\alpha dx\Big\|_0^2
+\Big\|\frac{1}{|\Omega|}\int\theta_{\beta}^\alpha dx\Big\|_0^2 +2\|\theta_{\beta}^\alpha\|_0^2 \nonumber\\[1mm]
&& \le C\epsilon\big(\|\bv_{\beta}^\alpha\|_0^2\|U\|_1^2+
\|\eta_{\beta}^\alpha\|_0^2\|\tilde\theta\|_2^2+\|\bv_{\beta}^\alpha\|_1^2+\|\tbv\|_1^4\big)\nonumber\\[1mm]
&& \quad +C\epsilon^2\big[\|\tilde F\|_{-1}^2+\|\eta_{\beta}^\alpha\|_0^2(\|\f\|_1^2+ \|U+\tbv\|_2^4)\big]\nonumber\\[1mm]
&& \quad +\|\eta_{\beta}^\alpha\|_0^2 C\big[\|\tilde\theta\|_2+\epsilon(\|\f\|_1 +\|U+ \tbv\|_2^4)+\epsilon\|\tbv\|_3 \nonumber\\[1mm]
&& \quad +\epsilon\|\tilde\theta\|_2^2(\|U+\tbv\|_1^2+\|\tbv\|_2^2)+\|\tbv\|_2^2\big]
+C\epsilon\|\eta_{\beta}^\alpha\|_0 \|P\|_2 \|U+\tbv\|_2  \nonumber\\[1mm]
&& \quad +C\epsilon(\|\tilde{F^{'}}\|_{-1}^2 +\|\tilde{G^{'}}\|_{-1}^2)
 +C\|\tbv\|_1^4 +C\|\tbv\|_1^2 \|\tilde\theta\|_1^2.   \label{djy7}
\end{eqnarray}

If we combine \eqref{djy4} with \eqref{djy7}, we see that there is a small constant $\epsilon_1>0$, such that for any $\epsilon\le\epsilon_1$
and sufficiently small $\|(\tbv,\tilde\theta)\|_3$,
\stepcounter{linear} \begin{eqnarray}\label{jjysosn}
&&\|\eta_{\beta}^\alpha\|_{\bar{H}^0}^2+
\|\bv_{\beta}^\alpha\|_0^2+\|\theta_{\beta}^\alpha\|_1^2+\beta\|\nabla \eta_{\beta}^\alpha\|_0\le C_1,  \label{djy8}
\end{eqnarray}
where $C_1$ is a positive constant independent of $\epsilon$, ${\alpha}$ and $\beta$.
Applying the elliptic regularity theory
to \eqref{CLa}$_1$, we have
$$\|\eta_{\beta}^\alpha\|_2 \leq C(\beta)\|(\mathrm{div}R_1(\eta_{\beta}^\alpha,\mathbf{v}_{\beta}^\alpha),
R_1(\eta_{\beta}^\alpha,\mathbf{v}_{\beta}^\alpha))\|_0\leq C\equiv C(\beta,C_1),$$
where the constant $C$ is independent of $\alpha$. Therefore, by the embedding theorem $H^2\hookrightarrow L^\infty$, we see that $$1_{\alpha}({\eta_{\beta}^\alpha})={\eta_{\beta}^\alpha},$$
if $\alpha$ is sufficiently small while $\beta$ is fixed.

With the help of this fact and the uniform estimate \eqref{jjysosn}, we can take limit as $\alpha\to 0$
(while keeping $\beta$ fixed) in \eqref{CLa} and use the standard weak convergence arguments to obtain a weak solution
$(\eta_{\beta},\mathbf{v}_{\beta},\theta_{\beta})$ of the problem \eqref{CLa} with $1$ in place of $1_\alpha$, and moreover
by (\ref{djy8}), we have
$$\|\eta_{\beta}\|_0^2+\|\bv_{\beta}\|_1^2+\|\theta_{\beta}\|_1^2+\beta\|\nabla \eta_{\beta}\|_0 \le C_1.$$
Recalling that the above estimate is uniform in $\beta$, we can thus take limit once more as $\beta\to 0$ in \eqref{CLa}
with $1$ in place of $1_\alpha$ to get a weak solution $(\eta, \mathbf{v},\theta)$ of the original problem \eqref{CL},
which is the weak limit of $(\eta_{\beta}, \mathbf{v}_{\beta},\theta_{\beta})$
and enjoys the estimate (i.e. \eqref{jjysosn}):
\stepcounter{linear}
\begin{equation}
\|\eta\|_0^2+\|\bv \|_1^2+\|\theta\|_1^2 \le C_1 . \label{226}
\end{equation}
Keeping in mind that the constant $C_1$ in \eqref{226} depends only on $\|(\tilde{F}',\tilde{G}')\|_{-1}$ (cf. \eqref{jjysosn})
and $L^2$ is dense in $H^{-1}$, we can thus obtain a weak solution $(\eta,\mathbf{v},\theta)$ of \eqref{CL}
for $\tilde{F}'$, $\tilde{G}'\in {H^{-1}}$ by a density argument. We remark that
the obtained solution $(\eta, \mathbf{v},\theta)$ satisfies the weak form of \eqref{CL}$_1$, i.e.,
$$\int_\Omega\left[ -\mathrm{div}(\varphi U)\eta + \left
(\dis{\frac{{\rm div}\bv}{\epsilon} }
+\tilde \bv\cdot\nabla\eta+\eta{\rm div}\tbv\right)\varphi\right] dx
=-\int_\Omega\epsilon{\rm div}(P(U+\tilde \bv))\varphi dx,\quad\forall\, \varphi\in H^1. $$
Moreover, extending $(\eta,P,U,\mathbf{v},\tilde{\mathbf{v}})$ outside of $\Omega$ by zero and still denoting the extended functions by $(\eta,P,U,\mathbf{v},\tilde{\mathbf{v}})$, we find that the above identity still holds in ${\mathbb R}^3$, i.e.,
\stepcounter{linear}
\begin{equation}  \label{new0226}
\int_{\mathbb{R}^3}\left[ -\mathrm{div}(\varphi U)\eta + \left(\dis{\frac{{\rm div}\bv}{\epsilon} }
+\tilde \bv\cdot\nabla\eta+\eta{\rm div}\tbv\right)\varphi\right] dx
=-\int_{\mathbb{R}^3}\epsilon{\rm div}(P(U+\tilde \bv))\varphi dx\end{equation}
for any $\varphi\in H^1(\mathbb{R}^3)$. This property is important in the proof of uniqueness in the next step.
\vspace{2mm}

\emph{Step 3: Uniqueness of weak solutions to \eqref{CL}.}

 Letting $(\eta_1,\bv_1,\theta_1)$, $(\eta_2,\bv_2,\theta_2)\in\bar{H}^0\times{H}_0^1\times{H}_0^1$ be
two weak solutions of the problem \eqref{CL} and denoting $\overline{\eta}=\eta_1-\eta_2$, $\overline{\bv}=\bv_1-\bv_2$,
$\overline{\theta}=\theta_1-\theta_2$, we find that
$(\overline{\eta},\ \overline{\bv},\ \overline{\theta})$ is a weak solution
of the following boundary value problem:
\stepcounter{linear}
\begin{equation}
\left\{\begin{array}{llll}
U\cdot\nabla \overline{\eta} + \dis{\frac{{\rm div}\overline{\bv}}{\epsilon} }
+\tilde \bv\cdot\nabla\overline{\eta}+\overline{\eta}{\rm div}\tbv =0,\\[1mm]
U\cdot\nabla \overline{\bv}-\mu\triangle \overline{\bv} -\zeta\nabla{\rm div}\overline{\bv}
+ \dis{\frac{\nabla \overline{\eta}+\nabla\overline{\theta}}{\epsilon} }+\tilde \theta\nabla\overline{\eta}
+\overline{\eta}\nabla\tilde \theta
=\epsilon (\overline{\eta} \f+\overline{\eta}(U+\tilde\bv)\cdot\nabla(U+\tbv)), \\[2mm]
U\cdot\nabla\overline{\theta}-\kappa\triangle\overline{\theta}
+ \dis{\frac{{\rm div}\overline{\bv}}{\epsilon} }+\overline{\eta}{\rm div}\tbv
=\epsilon (-\overline{\eta}(U+\tbv)\cdot\nabla\tilde\theta-\overline{\eta}\tilde\theta{\rm div}\tbv),\\[1mm]
\overline{\bv}=0,\ \ \ \ \overline{\theta}=0\ \  {\rm on}\ \ \partial\Omega\ \ \ {\rm and}
\quad \dis{\int_\Omega \overline{\eta}=0}.
\end{array}\right.  \label{jss}
\end{equation}

We want to test (\ref{jss}) with $\overline{\eta}$ which is unfortunately not in $H^1$. To circumvent
this difficulty, we use the technique of mollifiers. For a function $w$ of $x$ ($x\in\Omega$) we denote
$$(w)_\delta (x)=\int_\Omega\eta_\delta (x-y)w(y)dy,\ \ \ \ x\in\Omega ,$$
where $\eta_\delta$ is the Friedrichs mollifier.

Now, we take the test function $\varphi$ in \eqref{new0226} to be the mollifier to find that
the equation \eqref{jss}$_1$ can be regularized as
\stepcounter{linear}
\begin{equation}  \label{3CL}
\begin{aligned}
& U\cdot\nabla (\overline{\eta})_\delta +  \dis{\frac{({\rm div}\overline{\bv} )_\delta }{\epsilon}}
+\tilde \bv\cdot\nabla(\overline{\eta})_\delta +  (\overline{\eta} {\rm div}\tbv)_\delta  \\[1mm]
& \quad =\big( U\cdot\nabla (\overline{\eta})_\delta -(U\cdot\nabla \overline{\eta})_\delta \big)
+\big(\tilde \bv\cdot\nabla(\overline{\eta})_\delta - (\tilde \bv\cdot\nabla\overline{\eta})_\delta \big),\quad \mbox{ a.e. in }\Omega ,
\end{aligned}
\end{equation}
where $U\cdot\nabla\overline{\eta}=\mathrm{div}(U\overline{\eta})-\overline{\eta}\mathrm{div}U$ and $\tilde \bv\cdot\nabla
\overline{\eta}=\mathrm{div}(\tilde{\bv}\overline{\eta})-\eta\mathrm{div}\tilde{\bv}$.

Since $\eta\in L^2$, and $U,\tilde{\mathbf{v}}\in H_0^1\cap W^{1,\infty}$ by Sobolev's imbedding theorem,
we have by Friedrichs' lemma on commutators \cite[Lemma 3.1]{NS04} that
\stepcounter{linear}
\begin{equation}  \label{mmc}
U\cdot\nabla (\overline{\eta})_\delta -(U\cdot\nabla \overline{\eta})_\delta, \quad
 \tilde \bv\cdot\nabla(\overline{\eta})_\delta -(\tilde \bv\cdot\nabla
\overline{\eta})_\delta\to 0\;\;\mbox{ strongly in }L^2(\Omega).
\end{equation}
We should point here that \eqref{mmc} is shown in the whole space ${\mathbb R}^3$. However, if we extend $U$, $\tilde\bv$ and
$\eta$ outside of $\Omega$ by zero, we still have $\eta\in L^2({\mathbb R}^3)$ and
$U,\tilde{\mathbf{v}}\in H_0^1({\mathbb R}^3)\cap W^{1,\infty}({\mathbb R}^3)$, and hence we can apply Lemma 3.1 of \cite{NS04}
to get \eqref{mmc}.

Now, we test the equations \eqref{3CL}, \eqref{jss}$_2$ and \eqref{jss}$_3$ by
$(\overline{\eta})_\delta$, $\overline{\bv}$ and $\overline{\theta}$, respectively, to deduce that
\begin{eqnarray*}
&& \int\big(\mu|\nabla\overline{\bv}|^2+\zeta|\nabla\overline{\bv}|^2 + \kappa|\nabla\overline{\theta}|^2\big)dx   \\[1mm]
&& = \epsilon\int\overline{\eta}\overline{\bv}\Big[\f+(U+\tbv)\cdot\nabla(U+\tbv)\Big]
-\overline{\eta}\overline{\theta}\Big[(U+\tbv)\cdot\nabla\tilde\theta+\tilde\theta{\rm div}\tbv\Big]dx  \\[1mm]
&&  +\int\left[\frac12{\rm div}\tbv|(\overline{\eta})_\delta|^2+
\overline{\eta}\tilde{\theta}{\rm div}\bv
-(\overline{\eta} {\rm div}\tbv)_\delta (\eta)_\delta
-\overline{\eta}\overline{\theta}{\rm div}\tbv \right]dx \\[1mm]
&&  +\int\left\{\left(\frac{{\rm div}\overline{\bv}}{\epsilon} \eta
-\frac{({\rm div}\overline{\bv})_\delta}{\epsilon}(\eta)_\delta\right)+
[(U\cdot\nabla (\overline{\eta})_\delta -(U\cdot\nabla \overline{\eta})_\delta )
+(\tilde \bv\cdot\nabla(\overline{\eta})_\delta -(\tilde \bv\cdot\nabla
\overline{\eta})_\delta )](\eta)_\delta \right\}dx. 
\end{eqnarray*}
Letting $\delta\to 0$ in the above identity, using \eqref{mmc} and noting that (cf. \eqref{mmc})
$$ \int\left(\frac12{\rm div}\tbv|(\overline{\eta})_\delta |^2
-(\overline{\eta} {\rm div}\tbv)_\delta (\eta)_\delta
+\frac{{\rm div}\overline{\bv}}{\epsilon} \eta
-\frac{({\rm div}\overline{\bv})_\delta }{\epsilon}(\eta)_\delta\right)dx \to 0, $$
we conclude that
\stepcounter{linear}
\begin{eqnarray}
&&\int\big(\mu|\nabla\overline{\bv}|^2 +\zeta|\nabla\overline{\bv}|^2 +\kappa|\nabla\overline{\theta}|^2\big)dx
\le C(\|\overline{\eta}\|_0^2+\|\overline{\theta}\|_0^2)
 \big[(\|U\|_2+\|\tbv\|_2)\|\tilde\theta\|_3+\|\tbv\|_3(1+\|\tilde\theta\|_2)\big] \nonumber \\
&& \qquad +\epsilon C(\|\overline{\eta}\|_0^2+\|\overline{\bv}\|_0^2)(\|\f\|_2+\|U\|_2+\|\tbv\|_2).  \label{ECL}
\end{eqnarray}

On the other hand, for the inhomogeneous Stokes problem:
\begin{equation}\nonumber
\left\{\begin{array}{llll}
-\mu\triangle \overline{\bv}+ \dis{\frac{\nabla \overline{\eta}+\nabla\overline{\theta}}{\epsilon} }
=-U\cdot\nabla \overline{\bv}+\zeta\nabla{\rm div}\overline{\bv}+\epsilon \Big[\overline{\eta} \f
+\overline{\eta}(U+\tilde\bv)\cdot\nabla(U+\tbv)\Big]-(\tilde \theta\nabla\overline{\eta} +\overline{\eta}\nabla\tilde \theta),\\[1mm]
{\rm div}\overline\bv={\rm div}\overline\bv,\\[1mm]
\overline{\bv}=0,
\end{array}\right.
\end{equation}
we have the estimate
\stepcounter{linear}
\begin{eqnarray}
&& \epsilon\|\overline{\bv}\|_1+\|\overline{\eta}+\overline{\theta}\|_0\nonumber\\[1mm]
&& \quad \le \epsilon C\Big\{(\|U\|_2+1)\|\overline{\bv}\|_1+\|\overline{\bv}\|_1+\epsilon\|\overline{\eta}\|_0\Big[\|\f\|_2+(\|U\|_2+\|\tbv\|_2)^2\Big]
+\|\overline{\eta}\|_0\|\tilde\theta\|_2\Big\}.\label{ESP}
\end{eqnarray}

Combining \eqref{ECL} with \eqref{ESP}, making use of Poincar\'e's inequality, and recalling the smallness of $\epsilon$, $\|\tbv\|_3$ and
$\|\tilde\theta\|_3$, we find that
\stepcounter{linear}
\begin{eqnarray}
&&\|\overline{\eta}\|_0+\|\overline{\bv}\|_1+\|\overline{\theta}\|_1\le 0,\label{unique}
\end{eqnarray}
which implies $\eta_1=\eta_2,\ \bv_1=\bv_2,\ \theta_1=\theta_2$. This completes the proof of Lemma 2.3.
\end{proof}

In order to get higher order uniform-in-$\epsilon$ estimates of the $(\eta,\bv,\theta)$,
we have to bound $\|\bv\|_1$ and $\|\theta\|_1$ as shown in the following lemma.
\begin{lemma}\label{lmc1}
Let $(U,P)$ be the solution of \eqref{i1l} given in Lemma \ref{lmI1}. Let $\tilde F, \tilde G\in H^{-1}$.
Then we have the following uniform-in-$\epsilon$ estimate:
\stepcounter{linear}
\begin{equation}\label{lmc1-i}
\begin{aligned}
\|\bv\|_1^2+\|\theta\|_1^2 \le & C_5 \Big[\|\tbv\|_3\|\eta\|_0^2 +\epsilon^2(\|\tilde F\|_{-1}^2
+\|\tilde G\|_{-1}^2)+\epsilon\|P\|_2(\|U\|_2+\|\tbv\|_2)\|\eta\|_0\\[1mm]
& +\|\tbv\|_1^4+\|\eta\|_1^2\|\tilde\theta\|_1^2 +\|\tbv\|_1^2(\|\tilde\theta\|_1^2+\|\eta\|_1^2)\Big],
\end{aligned}
\end{equation}
where the constant $C_5>0$ is independent of $\epsilon$.
\end{lemma}
\begin{proof}
Multiplying $\eqref{CL}_1$, $\eqref{CL}_2$ and $\eqref{CL}_3$ by $\eta,\ v$
and $\theta$ in $L^2$ respectively, and summing up the resulting equations, we find that
\begin{eqnarray}
&& C^{'}(\|\bv\|_1^2+\|\theta\|_1^2)   \nonumber \\[1mm]
&& = -\frac{1}{\epsilon}\int\big({\rm div}\bv(\eta+\theta)+v\cdot(\nabla\eta+\nabla\theta)\big)dx
-\int\big[(U+\tilde \bv)\cdot\nabla\eta\cdot\eta+\eta^2{\rm div}\tbv
+\epsilon {\rm div}(P(U+\tilde \bv))\eta\big]dx \nonumber \\[1mm]
&& \quad +\int(\epsilon\tilde F-\tilde \bv\cdot\nabla\tilde \bv
-\tilde\theta\nabla\eta-\eta\nabla\tilde\theta)\cdot \bv dx
+\int\big(\epsilon \tilde G-\tilde \bv\cdot\nabla\tilde\theta
-(\eta+\tilde\theta){\rm div}\tbv\big)\theta dx \nonumber \\[2mm]
&& \le \frac{1}{2}\|\tbv\|_3\|\eta\|_0^2+\delta(\|\bv\|_1^2+\|\theta\|_1^2)
+C_{\delta}[\epsilon^2(\|\tilde F\|_{-1}^2+\|\tilde G\|_{-1}^2)+\epsilon\|P\|_2(\|U\|_2
+\|\tbv\|_2)\|\eta\|_0 \nonumber \\[1mm]
\stepcounter{linear}
&& \quad +\|\tbv\|_1^4+\|\eta\|_1^2\|\tilde\theta\|_1^2+\|\tbv\|_1^2(\|\tilde\theta\|_1^2
+\|\eta\|_1^2)],    \label{js3}
\end{eqnarray}
where we have used integration by parts, Sobolev's inequality and the fact that
$$-\frac{1}{\epsilon}\int\big((\eta+\theta){\rm div}\bv +\bv\cdot(\nabla\eta+\nabla\theta)\big)dx
=\frac{1}{\epsilon}\int\big[(\eta+\theta){\rm div}\bv -(\eta+\theta){\rm div}\bv\big]dx=0.$$

Finally, if we take $\delta$ in (\ref{js3}) suitably small and apply Poincar\'e's inequality,
we obtain the estimate \eqref{lmc1-i}.  \end{proof}

\subsubsection{Stokes problem}

We rewrite the momentum equations $\eqref{CL}_2$
as an inhomogeneous Stokes problem to derive the desired bounds for $\|\bv\|_3$
and $\Big\|\dis{\frac{\nabla(\eta+\theta)}{\epsilon}}\Big\|_{1}$:
\stepcounter{linear}
\begin{equation}
\label{stokes}
\left\{\begin{array}{llll}
-\mu\triangle \bv+ \dis{\frac{\nabla \eta+\nabla\theta}{\epsilon} }
=\epsilon \tilde F-\tilde \bv\cdot\nabla \tilde \bv-\tilde \theta\nabla\eta -\eta\nabla\tilde \theta-U\cdot\nabla \bv+\zeta\nabla{\rm div}\bv,\\[1mm]
{\rm div}\,\bv ={\rm div}\,\bv ,\\[1mm]
\bv =0,\ \ \ \ \ \mbox{on}\ \Omega.
\end{array}\right.
\end{equation}

By the usual estimates for the steady Stokes problem (cf. Galdi's book \cite[Chapter IV]{Galdi94}),
Sobolev's embedding $H^2\hookrightarrow L^\infty$ and the inequality
\stepcounter{linear}
\begin{equation}\label{yi}
\|{\rm div}\bv\|_1^2\le\delta\|\bv\|_3^2+C_\delta\|\bv\|_1^2,
\end{equation}
we have
\stepcounter{linear}
\begin{equation}\label{EoS0}
\begin{aligned}
\|\bv\|_{2}+\Big\|\frac{\nabla\eta+\nabla\theta}{\epsilon}\Big\|_0
\le & C\big(\|\epsilon\tilde F\|_0+\|\tilde \bv\cdot\nabla\tilde \bv\|_0+\|\tilde\theta\nabla\eta\|_0
+\|\eta\nabla\tilde\theta\|_0\\[1mm]
&+\|U\cdot\nabla \bv\|_0+\|{\rm div}\bv\|_{1}\big).
\end{aligned}
\end{equation}
and
\begin{equation}\nonumber
\begin{aligned}
\|\bv\|_{3}+\Big\|\frac{\nabla\eta+\nabla\theta}{\epsilon}\Big\|_1
\le & C\big( \|\epsilon\tilde F\|_1+\|\tilde \bv\cdot\nabla\tilde \bv\|_1+\|\tilde\theta\nabla\eta\|_1
+\|\eta\nabla\tilde\theta\|_1\\[1mm]
&+\|U\cdot\nabla \bv\|_1+\|{\rm div}\bv\|_{2}\big),
\end{aligned}
\end{equation}
which together with \eqref{lmc1-i}, \eqref{yi} and \eqref{EoS0} yields
\stepcounter{linear}
\begin{equation}\label{EoS1}
\begin{aligned}
& \|\bv\|_{3}+\Big\|\frac{\nabla\eta+\nabla\theta}{\epsilon}\Big\|_1
 \le C_6(1+\|U\|_2^2)\Big\{\epsilon(\|\tilde F\|_1+\|\tilde G\|_{-1})+\|\tilde \bv\|_3^2
+\|\tilde\theta\|_3\|\eta\|_2 \\[1mm]
&\quad +\|\tilde\bv\|_3^{1/2}\|\eta\|_0 +\Big[\epsilon\|P\|_2(\|U\|_2
+\|\tilde\bv\|_2)\|\eta\|_0\Big]^{1/2}
+\|\tbv\|_1(\|\tilde\theta\|_1+\|\eta\|_1)\Big\}+\|\nabla^2{\rm div}\bv\|_{0},
\end{aligned}
\end{equation}
where $C_6$ is a positive constant independent of $\epsilon$.

\subsubsection{Estimate of $\|\nabla^2{\rm div}\bv\|_0$}
As in \cite{Valli83,VZ86,JO11}, in order to control the term $\|\nabla^2{\rm div}\bv\|_0$ we
divide it into the interior part and the part near the boundary. We remark that here
we have to carefully deal with the terms which involve with the large parameter $1/\epsilon$ in \eqref{CL}.
\\[2mm]
{\sl I. Interior estimate}

First, we derive the interior estimate of $\nabla^2{\rm div}\bv$ by using the estimate \eqref{EoS0}.
Let $\chi_0$ be a $C_0^\infty$-function, then we have
\begin{lemma}\label{lmc2}
There is a positive constant $C_7$ independent of $\epsilon$, such that
\stepcounter{linear}
\begin{equation}\label{lmc2-i}
\begin{aligned}
&\mu\|\chi_0\nabla^2 \bv\|_0^2+\zeta\|\chi_0\nabla{\rm div}\bv\|_0^2+\kappa\|\chi_0\nabla^2\theta\|_0^2\\[1mm]
\le&C_7 \Big[(\|U\|_3+\|\tilde \bv\|_3)\|\eta\|_1^2+\epsilon^2(\|\tilde F\|_0^2+\|\tilde G\|_0^2)
+\epsilon\|P\|_2(\|U\|_2+\|\tilde \bv\|_2)\|\eta\|_1+\|\tilde \bv\|_2^4\\[1mm]
&+\|\eta\|_2^2\|\tilde\theta\|_2^2+\|U\|_3(\|\bv\|_1^2+\|\theta\|_1^2)
+\|\tilde \bv\|_2^2(\|\tilde\theta\|_1^2+\|\eta\|_1^2)+\|\bv\|_1^2\Big]
+\delta\Big\|\frac{\nabla\eta+\nabla\theta}{\epsilon}\Big\|_0^2.
\end{aligned}
\end{equation}
\end{lemma}
\begin{proof}
We differentiate \eqref{CL} with respect to $x$ to get that
\stepcounter{linear}
\begin{equation}
\label{CLx}
\left\{\begin{array}{llll}
U^j\partial_{ij}^2 \eta+ \dis{\frac{\partial_i{\rm div}\bv}{\epsilon}}
=-\tilde v^j\partial_{ij}^2\eta-\partial_i(U^j+\tilde v^j)\partial_j\eta
-\tilde v^j\partial_{ij}^2\eta-\partial_i(\eta{\rm div}\tilde \bv)
-\epsilon\partial_i{\rm div}(P(U+\tilde \bv)),\\[1mm]
U^j\partial_{ij}^2 v^k+\partial_iU^j\partial_{j} v^k-\mu\partial_{ijj}^3 v^k-\zeta\partial_{ik}^2{\rm div}\bv
+\dis{\frac{\partial_{ik}^2 \eta+\partial_{ik}^2\theta}{\epsilon} }
=\epsilon \partial_{i}\tilde F^k -\partial_{i}(\tilde v^j\partial_{j}\tilde v^k)
-\partial_{ik}^2(\tilde \theta\eta),\\[1mm]
U^j\partial_{ij}^2 \theta+\partial_iU^j\partial_{j}\theta-\kappa\partial_{ijj}^3\theta
+\dis{\frac{\partial_{i}{\rm div}\bv}{\epsilon} }
=\epsilon \partial_{i}\tilde G-\partial_{i}(\tilde v^j\partial_{j}\tilde \theta)
-\partial_{i}((\eta+\tilde \theta){\rm div}\tilde \bv).
\end{array}\right.
\end{equation}
Multiplying $\eqref{CLx}_1$, $\eqref{CLx}_2$ and $\eqref{CLx}_3$ by $\chi_0^2\partial_i\eta,\ \chi_0^2\partial_iv^k$
and $\chi_0^2\partial_i\theta$ in $L^2$ respectively, and summing up the resulting equations, we find that
\begin{eqnarray}
&&\mu\|\chi_0\partial_{ij}^2 v^k\|_0^2+\zeta\|\chi_0\partial_{i}{\rm div}\bv\|_0^2
+\kappa\|\chi_0\partial_{ij}^2\theta\|_0^2\nonumber\\[1mm]
=&&-\frac{1}{\epsilon}\int\chi_0^2\partial_{i}{\rm div}\bv(\partial_{i}\eta+\partial_{i}\theta)
+\chi_0^2\partial_{i}v^k\partial_{k}(\partial_{i}\eta+\partial_{i}\theta)dx\nonumber\\[1mm]
&&-\int(2\mu\chi_0\partial_j\chi_0\partial_{ij}^2 v^k\partial_{i} v^k
+2\zeta\chi_0\partial_k\chi_0\partial_{i}{\rm div} \bv\partial_{i} v^k
 +2\kappa\chi_0\partial_j\chi_0\partial_{ij}^2 \theta\partial_{i}\theta)dx\nonumber\\[1mm]
&&-\int\chi_0^2\Big[\partial_i(U^j+\tilde v^j)\partial_{j}\eta+(U^j+\tilde v^j)\partial_{ji}^2\eta
+\partial_i\eta{\rm div}\tilde \bv  +\eta\partial_i{\rm div}\tilde \bv
+\epsilon \partial_i{\rm div}(P(U+\tilde \bv))\Big]\partial_i\eta dx\nonumber\\[1mm]
&&+\int\chi_0^2(\epsilon\partial_i\tilde F^k-\partial_i(\tilde \bv\cdot\nabla\tilde v^k)
-\partial_{ik}^2(\tilde \theta\eta))\partial_{i}v^k dx\nonumber\\[1mm]
&&+\int\chi_0^2(\epsilon \partial_{i}\tilde G-\partial_{i}(\tilde v^j\partial_{j}\tilde \theta)
-\partial_{i}((\eta+\tilde \theta){\rm div}\tilde \bv))\partial_{i}\theta dx.\nonumber
\end{eqnarray}
If we apply partial integrations to the above identity, employ Sobolev's and Young's inequalities
and the fact that
$$
\begin{aligned}
&-\frac{1}{\epsilon}\int\chi_0^2\partial_{i}{\rm div}\bv(\partial_{i}\eta+\partial_{i}\theta)
+\chi_0^2\partial_{i}v^k\partial_{k}(\partial_{i}\eta+\partial_{i}\theta)dx\\[1mm]
& =\frac{1}{\epsilon}\int2\chi_0\partial_{k}\chi_0\partial_{i}v^k(\partial_{i}\eta+\partial_{i}\theta)dx
+\frac{1}{\epsilon}\int\chi_0^2\partial_{i}v^k\partial_{k}(\partial_{i}\eta+\partial_{i}\theta)
-\chi_0^2\partial_{i}v^k\partial_{k}(\partial_{i}\eta+\partial_{i}\theta)dx\\[1mm]
& =\frac{1}{\epsilon}\int2\chi_0\partial_{k}\chi_0\partial_{i}v^k(\partial_{i}\eta+\partial_{i}\theta)dx
\le\delta\Big\|\frac{\nabla\eta+\nabla\theta}{\epsilon}\Big\|_0^2+C_\delta\|\bv\|_1^2,
\end{aligned}$$
we infer by summing up $i,j,k$ that
\begin{eqnarray}
&&\mu\|\chi_0\nabla^2 \bv\|_0^2+\zeta\|\chi_0\nabla{\rm div}\bv\|_0^2+\kappa\|\chi_0\nabla^2\theta\|_0^2\nonumber\\[1mm]
&& \le (\|U\|_3+\|\tilde \bv\|_3)\|\eta\|_1^2+\delta\Big(\|\bv\|_2^2
+\|\theta\|_2^2+\Big\|\frac{\nabla\eta+\nabla\theta}{\epsilon}\Big\|_0^2\Big)\nonumber\\[1mm]
&&\quad +C_\delta\Big[\epsilon^2(\|\tilde F\|_0^2+\|\tilde G\|_0^2)
+\epsilon\|P\|_2(\|U\|_2+\|\tilde \bv\|_2)\|\eta\|_1
+\|\tilde \bv\|_2^4+\|\eta\|_2^2\|\tilde\theta\|_2^2\nonumber\\[1mm]
&& \quad +\|U\|_3(\|\bv\|_1^2+\|\theta\|_1^2)
+\|\tilde \bv\|_2^2(\|\tilde\theta\|_1^2+\|\eta\|_1^2)+\|\bv\|_1^2\Big],\nonumber
\end{eqnarray}
which, by using Poincar\'e's inequality and choosing $\delta$ appropriately small, implies the lemma.
\end{proof}
\begin{lemma}\label{lmc3}
There is a positive constant $C_8$ independent of $\epsilon$, such that
\stepcounter{linear}
\begin{equation}\label{lmc3-i}
\begin{aligned}
&\mu\|\chi_0\nabla^3 \bv\|_0^2+\zeta\|\chi_0\nabla^2{\rm div}\bv\|_0^2+\kappa\|\chi_0\nabla^3\theta\|_0^2\\[1mm]
& \le C_8\Big[(\|U\|_3+\|\tilde \bv\|_3)\|\eta\|_2^2+\epsilon^2(\|\tilde F\|_1^2+\|\tilde G\|_1^2)
+\epsilon\|P\|_3(\|U\|_3+\|\tilde \bv\|_3)\|\eta\|_2+\|\tilde \bv\|_3^4\\[1mm]
& \quad +\|\eta\|_2^2\|\tilde\theta\|_2^2+\|U\|_3(\|\bv\|_2^2+\|\theta\|_2^2)
+\|\tilde \bv\|_2^2(\|\tilde\theta\|_2^2+\|\eta\|_2^2)+\|\bv\|_2^2\Big]
+\delta\Big\|\frac{\nabla\eta+\nabla\theta}{\epsilon}\Big\|_1^2.
\end{aligned}
\end{equation}
\end{lemma}
\begin{proof}
We differentiate \eqref{CL} twice with respect to $x$ to get that
\stepcounter{linear}
\begin{equation}
\label{CLxx}
\left\{\begin{array}{llll}
U^j\partial_{ilj}^3 \eta +\dis{\frac{\partial_i{\rm div}\bv}{\epsilon} } =-\tilde v^j\partial_{ilj}^3\eta
-\partial_{il}^2(U^j+\tilde v^j)\partial_j\eta -\partial_i(U^j+\tilde v^j)\partial_{jl}^2\eta\\[1mm]
\ \ \ \ \ \ \ \ \ \ \ \ \ \ \ \ \ \ \ \ \ \ \ \
-\partial_l(U^j+\tilde v^j)\partial_{ij}^2\eta -\partial_{il}(\eta{\rm div}\tilde \bv)
-\epsilon\partial_{il}^2{\rm div}(P(U+\tilde \bv)),\\[1mm]
U^j\partial_{ijl}^3 v^k+\partial_lU^j\partial_{ij}^2 v^k
+\partial_{il}^2U^j\partial_{j} v^k+\partial_iU^j\partial_{lj}^2 v^k
-\mu\partial_{iljj}^4 v^k-\zeta\partial_{ilk}^3{\rm div}\bv
+\dis{\frac{\partial_{ilk}^3 \eta +\partial_{ilk}^3\theta}{\epsilon} }\\[1mm]
\ \ \ \ \ \ \ \ \ \ \ \ \ \ \ \ \ \ \ \ \
=\epsilon \partial_{il}^2\tilde F^k-\partial_{il}^2(\tilde v^j\partial_{j}\tilde v^k)
-\partial_{ilk}^3(\tilde \theta\eta),\\[1mm]
U^j\partial_{ijl}^3 \theta+\partial_lU^j\partial_{ij}^2 \theta+\partial_{il}^2U^j\partial_{j} \theta
+\partial_iU^j\partial_{lj}^2 \theta
-\kappa\partial_{iljj}^4\theta+\dis{\frac{\partial_{il}^2{\rm div}\bv}{\epsilon} }\\[1mm]
\ \ \ \ \ \ \ \ \ \ \ \ \ \ \ \ \ \ \ \ \
=\epsilon \partial_{il}^2\tilde G -\partial_{il}^2(\tilde v^j\partial_{j}\tilde \theta)
-\partial_{il}^2((\eta+\tilde \theta){\rm div}\tilde \bv).
\end{array}\right.
\end{equation}
Multiplying $\eqref{CLxx}_1$, $\eqref{CLxx}_2$ and $\eqref{CLxx}_3$ again
by $\chi_0^2\partial_{il}^2\eta,\ \chi_0^2\partial_{il}^2v^k$ and $\chi_0^2\partial_{il}^2\theta$ respectively,
and summing up the resulting equations, we deduce that
\begin{eqnarray}
&&\mu\|\chi_0\partial_{ilj}^3 v^k\|_0^2+\zeta\|\chi_0\partial_{il}^2{\rm div}\bv\|_0^2
+\kappa\|\chi_0\partial_{ilj}^3\theta\|_0^2\nonumber\\[1mm]
=&&-\frac{1}{\epsilon}\int\chi_0^2\partial_{il}{\rm div}\bv(\partial_{il}\eta
+\partial_{il}^2\theta)+\chi_0^2\partial_{il}^2v^k\partial_{k}(\partial_{il}^2\eta+\partial_{il}^2\theta)dx\nonumber\\[1mm]
&&-\int 2\chi_0\partial_j\chi_0(\mu\partial_{ilj}^3v^k\partial_{il}^2v^k+\kappa\partial_{ilj}^3\theta\partial_{il}^2\theta)
+2\zeta\chi_0\partial_k\chi_0\partial_{il}{\rm div}v\partial_{il}v^k dx\nonumber\\[1mm]
&&-\int\Big[ U^j\partial_{ilj}^3 \eta+\tilde v^j\partial_{ilj}^3\eta
+\partial_{il}^2(U^j+\tilde v^j)\partial_j\eta+\partial_i(U^j+\tilde v^j)\partial_{jl}^2\eta\nonumber\\[1mm]
&&\ \ \ \ \ \ +\partial_l(U^j+\tilde v^j)\partial_{ij}^2\eta +\partial_{il}^2
(\eta{\rm div}\tilde \bv)+\epsilon\partial_{il}^2{\rm div}(P(U+\tilde \bv))\Big]\chi_0^2\partial_{il}^2\eta dx\nonumber\\[1mm]
&& +\int\Big[ \epsilon \partial_{il}^2\tilde F^k-\partial_{il}^2(\tilde v^j\partial_{j}\tilde v^k)
-\partial_{ilk}^3(\tilde \theta\eta)
-(U^j\partial_{ijl}^3 v^k+\partial_lU^j\partial_{ij}^2 v^k\nonumber\\[1mm]
&&\ \ \ \ \ \ +\partial_{il}^2U^j\partial_{j} v^k+\partial_iU^j\partial_{lj}^2 v^k)\Big]\chi_0^2\partial_{il}^2v^k dx\nonumber\\[1mm]
&&+\int\Big[ \epsilon \partial_{il}^2\tilde G-\partial_{il}^2(\tilde v^j\partial_{j}\tilde \theta)
-\partial_{il}^2((\eta+\tilde \theta){\rm div}\tilde \bv)
-(U^j\partial_{ijl}^3 \theta+\partial_lU^j\partial_{ij}^2 \theta\nonumber\\[1mm]
&&\ \ \ \ \ \ +\partial_{il}^2U^j\partial_{j}\theta+\partial_iU^j\partial_{lj}^2\theta)\Big]\chi_0^2\partial_{il}\theta dx.\nonumber
\end{eqnarray}
We integrate by parts the above identity, utilize Sobolev's inequality and the fact that
\begin{eqnarray}
&&-\frac{1}{\epsilon}\int\chi_0^2\partial_{il}^2{\rm div}\bv(\partial_{il}^2\eta+\partial_{il}^2\theta)
+\chi_0^2\partial_{il}^2v^k\partial_{k}(\partial_{il}^2\eta+\partial_{il}^2\theta)dx\nonumber\\[1mm]
=&&\frac{1}{\epsilon}\int\chi_0^2\partial_{il}^2v^k\partial_k(\partial_{il}^2\eta+\partial_{il}^2\theta)
-\chi_0^2\partial_{il}^2v^k\partial_{k}(\partial_{il}^2\eta+\partial_{il}^2\theta)dx
+\frac{1}{\epsilon}\int2\chi_0\partial_k\chi_0\partial_{il}^2v^k(\partial_{il}^2\eta+\partial_{il}^2\theta)dx\nonumber\\[1mm]
=&&\frac{1}{\epsilon}\int2\chi_0\partial_k\chi_0\partial_{il}^2v^k(\partial_{il}^2\eta+\partial_{il}^2\theta)dx
\le\delta\Big\|\frac{\nabla\eta+\nabla\theta}{\epsilon}\Big\|_1^2+C_\delta\|\bv\|_2^2,\nonumber
\end{eqnarray}
and sum up $i,j,k$ to infer that
\begin{eqnarray*}
&&\mu\|\chi_0\nabla^3 \bv\|_0^2+\zeta\|\chi_0\nabla^2{\rm div}\bv\|_0^2+\kappa\|\chi_0\nabla^3\theta\|_0^2\\[1mm]
&& \le (\|U\|_3+\|\tilde \bv\|_3)\|\eta\|_2^2+\delta(\|\bv\|_3^2+\|\theta\|_3^2
+\Big\|\frac{\nabla\eta+\nabla\theta}{\epsilon}\Big\|_1^2)\\[1mm]
&& \quad +C_\delta[\epsilon^2(\|\tilde F\|_1^2+\|\tilde G\|_1^2)+\epsilon\|P\|_3(\|U\|_3+\|\tilde \bv\|_3)\|\eta\|_2
+\|\tilde \bv\|_3^4+\|\eta\|_2^2\|\tilde\theta\|_2^2\\[1mm]
&& \quad +\|U\|_3(\|\bv\|_2^2+\|\theta\|_2^2) +\|\tilde \bv\|_2^2(\|\tilde\theta\|_2^2+\|\eta\|_2^2)+\|\bv\|_2^2],
\end{eqnarray*}
which, by employing Poincar\'e's inequality and choosing $\delta$ suitably small, gives the lemma.
\end{proof}

\noindent {\sl II. Boundary estimate}

Next, we shall use the method of local coordinates to bound
$\nabla^2{\rm div}\bv$ in the vicinity of the boundary (also see \cite{Valli83,VZ86,JO11}).
For completeness, we briefly describe the local coordinates as follows.
First, one construct the local coordinates by the isothermal coordinates $\lambda(\varphi,\phi)$
to derive an estimate near the boundary (see also \cite{Valli83,VZ86}), where
$$\lambda_{\varphi}\cdot\lambda_{\varphi}>0,\ \ \lambda_{\phi}\cdot\lambda_{\phi}>0,
\ \ \lambda_{\varphi}\cdot\lambda_{\phi}=0.$$
The boundary $\partial\Omega$ can be covered by a finite number of bounded open sets
$W^k\subset {\mathbb R}^3$, $k =1,2,\cdots,L,$ such that for any $x\in W^k\cap\Omega$,
\stepcounter{linear}
\begin{equation}
x =\Lambda^k(\varphi,\phi, r)\equiv \lambda^k(\varphi,\phi) +r\mathbf{n}(\lambda^k(\varphi,\phi)),
\end{equation}
where $\lambda^k(\varphi,\phi)$ is the isothermal coordinates and $\mathbf{n}$
is the unit outer normal to $\partial\Omega$.

Without confusion, we will omit the superscript $k$ in each $W^k$ in the following.
We construct the orthonormal system corresponding to the local coordinates by
\stepcounter{linear}
\begin{equation}
e_1:=\frac{\lambda_{\varphi}}{|\lambda_{\varphi}|},\ \ e_2:=\frac{\lambda_{\phi}}{|\lambda_{\phi}|},\ \
e_3:=e_1\times e_2\equiv \mathbf{n}(\lambda).
\end{equation}
By a straightforward calculation, we see that for sufficiently small $r$, 
$$
\begin{aligned}
& C^2\ni J:={\rm det}\,{\rm Jac}\Lambda={\rm det}\frac{\partial x}{\partial(\varphi,\phi, r)}
=\Lambda_\varphi\times\Lambda_\phi\cdot e_3
\\=&|\lambda_\varphi||\lambda_\phi|+ r(|\lambda_\varphi|\mathbf{n}_\phi\cdot e_2
+ |\lambda_\phi|\mathbf{n}_\varphi\cdot e_1)
+ r^2[(\mathbf{n}_\varphi\cdot e_1)(\mathbf{n}_\phi\cdot e_2)
-(\mathbf{n}_\varphi\cdot e_2)(\mathbf{n}_\phi\cdot e_1)]> 0.
\end{aligned}
$$
And, we can easily derive the following relations as
$({\rm Jac}\Lambda^{-1})\circ\Lambda=({\rm Jac}\Lambda)^{-1}$ (also see \cite{Valli83}):
\stepcounter{linear}
\begin{eqnarray} \label{3.5.1}
&& [\nabla(\Lambda^{-1})^1]\circ\Lambda=J^{-1}(\Lambda_\phi\times e_3), \\
&&  \stepcounter{linear} \label{3.5.2}
[\nabla(\Lambda^{-1})^2]\circ\Lambda=J^{-1}(e_3\times \Lambda_\varphi),  \\
&& \stepcounter{linear}   \label{3.5.3}
[\nabla(\Lambda^{-1})^3]\circ\Lambda=J^{-1}(\Lambda_\varphi\times \Lambda_\phi)=e_3,
\end{eqnarray}
where the symbol $\circ$ stands for the composite of operators.
Set $y := (\varphi,\phi, r)$, and denote by $D_i$ the partial derivative
with respect to $y_i$ in local coordinates. We set the unknowns in local coordinates
$$\hat\eta(t, y):= \eta(t,\Lambda(y)),\quad {\hv}(t,y):= \bv(t,\Lambda(y)),\quad
\hat\theta(t,y):=\theta(t,\Lambda(y)),$$
and the knowns
$$\hat U(t,y):= U(t,\Lambda(y)),\quad \hat{\tilde \bv}(t,y):= \hat{\tilde \bv}(t,\Lambda(y)),\quad
\hat{\tilde\theta}(t,y):={\tilde\theta}(t,\Lambda(y)).$$
Then, we rewrite the system \eqref{CL} in $[0,T]\times\tilde{\Omega}$,
where $\tilde{\Omega}:= \Lambda^{-1}(W\cap\Omega)$, as follows.
\stepcounter{linear}
\begin{equation}
\label{LC}
\left\{\begin{array}{llll}
\hat{U}^ja_{kj}D_k \hat\eta + \dis{\frac{a_{kj}D_k\hat v^j}{\epsilon} }
=-\hat{\tilde v}^ja_{kj}D_k\hat\eta-\hat\eta a_{kj}D_k\hat{\tilde v}^j
-\epsilon a_{kj}D_k(\hat P(\hat U^j+\hat{\tilde v}^j)),\\[2mm]
\hat{U}^ja_{kj}D_k\hat v^i-\mu a_{kj}D_k(a_{lj}D_l\hat v^i)-\zeta a_{ki}D_k(a_{lj}D_l\hat v^j)
+ \dis{\frac{a_{ki}D_k(\hat \eta+\hat \theta)}{\epsilon} }\\
\quad\quad=\epsilon \hat{\tilde F}^i-\hat{\tilde v}^ja_{kj}D_k\hat{\tilde v}^i
-\hat{\tilde\theta} a_{ki}D_k\hat\eta -\hat\eta a_{ki}D_k\hat{\tilde\theta},\\[2mm]
\hat U^ja_{kj}D_k\hat\theta-\kappa a_{kj}D_k(a_{lj}D_l\hat\theta)+ \dis{\frac{a_{kj}D_k\hat v^j}{\epsilon} }
=\epsilon\hat{\tilde G}-\hat{\tilde v}^ja_{kj}D_k\tilde\theta -\hat\eta a_{kj}D_k\hat{\tilde v}^j
-\hat{\tilde\theta} a_{kj}D_k\hat{\tilde v}^j,
\end{array}\right.
\end{equation}
with boundary conditions
\stepcounter{linear}
\begin{align}
&{\hv}(t,y)=0,\ \ \hat\theta(t,y)=0\ \ \ {\rm on}\ \ \partial\tilde\Omega,\label{bc!}
\end{align}
where $a_{ij}$ is the $(i, j)$-th entry of the matrix ${\rm Jac}(\Lambda^{-1})=\frac{\partial y}{\partial x}$.
Clearly, $a_{ij}$ is a $C^2$-function, and it follows from (\ref{3.5.1})--(\ref{3.5.3}) that
\stepcounter{linear}
\begin{equation}\label{3.5.4}
\sum_{j=1}^3a_{3j} a_{3j} = |\mathbf{n}|^2 = 1,\ \
\sum_{j=1}^3a_{1j} a_{3j} =\sum_{j=1}^3a_{2j} a_{3j} = 0.
\end{equation}
Moreover, this localized system has the following properties (see also \cite{Valli83}):
\stepcounter{linear}
\begin{prop}\label{p3}
$D_i(J a_{ij}) = 0,\ {\rm for}\  j = 1, 2, 3;\ \varsigma D_{\tau}{\hv} = 0,\
\varsigma D_{\tau}D_{\xi}{\hv} = 0\ {\rm on}\ \partial\tilde{\Omega}$ in the tangential directions
$\tau, \xi = 1,2$, where $\varsigma\in C_0^\infty (\Lambda^{-1}(W))$. Similarly,
$\varsigma D_{\tau}\hat\theta= 0,\ \varsigma D_{\tau}D_{\xi}\hat\theta = 0\ {\rm on}\ \partial\tilde{\Omega}.$
\end{prop}

Recalling $D_j =\sum_{i=1}^3 a_{ji}\partial_i$, we will frequently make use of the following relations
without pointing out explicitly in subsequent calculations:
\begin{equation}\label{xy}
\|D_y {\hv}\|_{L^p(\Omega)}\le C\|\nabla_x \bv\|_{L^p(\Omega)},\ \ \|D_y^2{\hv}\|_{L^p(\Omega)}
\le C\|\nabla_x \bv\|_{W^{1,p}(\Omega)},\ 1\le p\le \infty.
\end{equation}
The above inequalities apply to $\eta$, $\theta$ and $U$, $\tilde \bv$, $\tilde\theta$, too.

By virtue of the interpolation $\|\cdot\|_{H^2}^2\le\delta \|\cdot\|_{H^3}^2+C_\delta\|\cdot\|_{H^1}^2$,
the boundary estimate of $\|\nabla^2{\rm div} \bu\|_{L_t^2(L^2)}$ can be reduced to the boundedness of
$$\int_0^t\int_{\tilde\Omega} J\chi^2|D_y^2(a_{ji}D_jU^i)|dyds,$$
where $\chi$ is a $C_0^\infty (\Lambda^{-1}(W))$-function.
So, we can split the estimate of derivatives on the boundary into two parts:
the estimate of derivatives in the tangential directions and in the normal direction.
\\[1mm]
{\sl Part 1. Estimate of derivatives in the tangential directions}

First, we apply $D_{\tau\xi}^2$ to \eqref{LC} with $\tau$, $\xi$ being the tangential directions to
$\partial\tilde\Omega$ to get
\stepcounter{linear}
\begin{equation}
\label{LCtx}
\left\{\begin{array}{llll}
\hat{U}^ja_{kj}D_{k\tau\xi}^3 \hat\eta + \dis{\frac{1}{\epsilon}D_{\tau\xi}^2[a_{kj}D_k\hat v^j] }\\[1mm]
=D_{\xi}({\hat U}^ja_{kj})D_{k\tau}^2\hat\eta+D_{\tau\xi}^2({\hat U}^ja_{kj})D_{k}\hat\eta
+D_{\tau}({\hat U}^ja_{kj})D_{k\xi}^2\hat\eta\\[1mm]
\quad-D_{\tau\xi}^2[\hat{\tilde v}^ja_{kj}D_k\hat\eta+\hat\eta a_{kj}D_k\hat{\tilde v}^j
+\epsilon a_{kj}D_k(\hat P(\hat U^j+\hat{\tilde v}^j))],\\[2mm]
\hat{U}^ja_{kj}D_{k\tau\xi}^3\hat v^i-D_{\tau\xi}^2[\mu a_{kj}D_k(a_{lj}D_l\hat v^i)
+\zeta a_{ki}D_k(a_{lj}D_l\hat v^j)]
+ \dis{\frac{1}{\epsilon}D_{\tau\xi}^2[a_{ki}D_k(\hat \eta+\hat \theta)] }\\[3mm]
=D_{\xi}({\hat U}^ja_{kj})D_{k\tau}^2{\hat v}^i+D_{\tau\xi}^2({\hat U}^ja_{kj})D_{k}{\hat v}^i
+D_{\tau}({\hat U}^ja_{kj})D_{k\xi}^2{\hat v}^i\\[1mm]
\quad+D_{\tau\xi}^2[\epsilon \hat{\tilde F}^i-\hat{\tilde v}^ja_{kj}D_k\hat{\tilde v}^i
-\hat{\tilde\theta} a_{ki}D_k\hat\eta -\hat\eta a_{ki}D_k\hat{\tilde\theta}],\\[2mm]
\hat U^ja_{kj}D_{k\tau\xi}^3\hat\theta-\kappa D_{\tau\xi}^2[a_{kj}D_k(a_{lj}D_l\hat\theta)]
+ \dis{\frac{1}{\epsilon}D_{\tau\xi}^2[a_{kj}D_k\hat v^j] }\\[1mm]
=D_{\xi}({\hat U}^ja_{kj})D_{k\tau}^2{\hat\theta}+D_{\tau\xi}^2({\hat U}^ja_{kj})D_{k}{\hat\theta}
+D_{\tau}({\hat U}^ja_{kj})D_{k\xi}^2{\hat\theta}\\[1mm]
\quad+D_{\tau\xi}^2[\epsilon\hat{\tilde G}-\hat{\tilde v}^ja_{kj}D_k\tilde\theta
-\hat\eta a_{kj}D_k\hat{\tilde v}^j -\hat{\tilde\theta} a_{kj}D_k\hat{\tilde v}^j].
\end{array}\right.
\end{equation}
We multiply $\eqref{LCtx}_1$, $\eqref{LCtx}_2$ and $\eqref{LCtx}_3$ by $J\chi^2D_{\tau\xi}\hat\eta$,
$J\chi^2D_{\tau\xi}{\hat v}^i$ and $J\chi^2D_{\tau\xi}\hat\theta$ respectively,
and integrate the resulting identities to deduce that
\stepcounter{linear}
\begin{eqnarray}
&& - \int_{\tilde\Omega}D_{\tau\xi}^2[\mu a_{kj}D_k(a_{lj}D_l\hat v^i)
+\zeta a_{ki}D_k(a_{lj}D_l\hat v^j)]\cdot J\chi^2D_{\tau\xi}^2{\hat v}^idy \nonumber \\[1mm]
&& -\int_{\tilde\Omega}\kappa D_{\tau\xi}^2[a_{kj}D_k(a_{lj}D_l\hat\theta)]\cdot J\chi^2D_{\tau\xi}^2{\hat\theta}dy
\nonumber \\[1mm]
&& +\int_{\tilde\Omega}\hat{U}^ja_{kj}(D_{k\tau\xi}^3 \hat\eta\cdot J\chi^2D_{\tau\xi}^2{\hat\eta}
+D_{k\tau\xi}^3 {\hat v}^i\cdot J\chi^2D_{\tau\xi}^2{\hat v}^i
+D_{k\tau\xi}^3 \hat\theta\cdot J\chi^2D_{\tau\xi}^2{\hat\theta})dy \nonumber  \\[1mm]
 &&+\dis{\frac{1}{\epsilon}}\int_{\tilde\Omega}\big\{ D_{\tau\xi}^2[a_{kj}D_k\hat v^j]\cdot
J\chi^2D_{\tau\xi}^2{\hat\eta}+D_{\tau\xi}^2[a_{ki}D_k(\hat \eta+\hat \theta)]\cdot J\chi^2D_{\tau\xi}^2{\hat v}^i
+D_{\tau\xi}^2[a_{kj}D_k\hat v^j]\cdot J\chi^2D_{\tau\xi}^2{\hat\theta}\big\}dy \nonumber \\[1mm]
=&& \int_{\tilde\Omega}\big\{D_{\xi}({\hat U}^ja_{kj})D_{k\tau}^2\hat\eta
+D_{\tau\xi}^2({\hat U}^ja_{kj})D_{k}\hat\eta+D_{\tau}({\hat U}^ja_{kj})D_{k\xi}^2\hat\eta \label{BVxxx} \\[1mm]
&& -D_{\tau\xi}^2[\hat{\tilde v}^ja_{kj}D_k\hat\eta+\hat\eta a_{kj}D_k\hat{\tilde v}^j
+\epsilon a_{kj}D_k(\hat P(\hat U^j+\hat{\tilde v}^j))]\big\}\cdot J\chi^2D_{\tau\xi}^2{\hat\eta}dy
\nonumber  \\[1mm]
&& +\int_{\tilde\Omega}\big\{ D_{\xi}({\hat U}^ja_{kj})D_{k\tau}^2{\hat v}^i
+D_{\tau\xi}^2({\hat U}^ja_{kj})D_{k}{\hat v}^i+D_{\tau}({\hat U}^ja_{kj})D_{k\xi}^2{\hat v}^i \nonumber \\[1mm]
&& +D_{\tau\xi}^2[\epsilon \hat{\tilde F}^i-\hat{\tilde v}^ja_{kj}D_k\hat{\tilde v}^i
-\hat{\tilde\theta} a_{ki}D_k\hat\eta -\hat\eta a_{ki}D_k\hat{\tilde\theta}]\big\}
\cdot J\chi^2D_{\tau\xi}^2{\hat v}^idy  \nonumber  \\[1mm]
&& +\int_{\tilde\Omega} \big\{ D_{\xi}({\hat U}^ja_{kj})D_{k\tau}^2{\hat\theta}
+D_{\tau\xi}^2({\hat U}^ja_{kj})D_{k}{\hat\theta}+D_{\tau}({\hat U}^ja_{kj})D_{k\xi}^2{\hat\theta}\nonumber  \\[1mm]
&& +D_{\tau\xi}^2[\epsilon\hat{\tilde G}-\hat{\tilde v}^ja_{kj}D_k\tilde\theta
-\hat\eta a_{kj}D_k\hat{\tilde v}^j -\hat{\tilde\theta} a_{kj}D_k\hat{\tilde v}^j]\big\}
\cdot J\chi^2D_{\tau\xi}^2{\hat\theta}dy.  \nonumber
\end{eqnarray}
Now, we denote LHS of \eqref{BVxxx} :=$L_1'+L_2'+L_3'+L_4'$ and have to deal with each term
due to integration by part and the boundary conditions.
\begin{eqnarray}
L_1'=&& -\int_{\tilde\Omega}\Big\{ D_{\tau\xi}^2(\mu a_{kj})D_k(a_{lj}D_l\hat v^i)
+D_\tau(\mu a_{kj})D_{k\xi}^2(a_{lj}D_l{\hat v}^i) \nonumber\\[1mm]
&&+D_\xi(\mu a_{kj})D_k[D_\tau(a_{lj})D_l{\hat v}^i+a_{lj}D_{l\tau}^2{\hat v}^i]
+\mu a_{kj}D_{k\xi}[D_\tau(a_{lj})D_l{\hat v}^i+a_{lj}D_{l\tau}^2{\hat v}^i] \nonumber\\[2mm]
&&+D_{\tau\xi}^2(\zeta a_{ki})D_k(a_{lj}D_l\hat v^j)+D_\tau(\zeta a_{ki})D_{k\xi}^2(a_{lj}D_l{\hat v}^j) \nonumber\\[2mm]
&&+D_\xi(\zeta a_{ki})D_k[D_\tau(a_{lj})D_l{\hat v}^j+a_{lj}D_{l\tau}^2{\hat v}^j] \nonumber\\[2mm]
&&+\zeta a_{ki}D_{k\xi}^2[D_\tau(a_{lj})D_l{\hat v}^j
+a_{lj}D_{l\tau}{\hat v}^j]\Big\}\cdot J\chi^2D_{\tau\xi}^2{\hat v}^idy\nonumber
\end{eqnarray}
where
\begin{eqnarray}
&&-\int_{\tilde\Omega} [\mu a_{kj}D_{k\xi}^2(a_{lj}D_{l\tau}^2{\hat v}^i)
+\zeta a_{ki}D_{k\xi}^2(a_{lj}D_{l\tau}^2{\hat v}^j)]\cdot J\chi^2D_{\tau\xi}^2{\hat v}^idy\nonumber\\[1mm]
=&&\int_{\tilde\Omega} \mu J\chi^2a_{kj}D_{k\tau\xi}^3{\hat v}^ia_{lj}
D_{l\tau\xi}^3{\hat v}^i+\zeta J\chi^2a_{ki}D_{k\tau\xi}^3{\hat v}^ia_{lj}D_{l\tau\xi}^3{\hat v}^jdy\nonumber\\[1mm]
&&+\int_{\tilde\Omega}[D_k(\mu J\chi^2a_{kj})D_{\xi}(a_{lj}D_{l\tau}{\hat v}^i)\cdot D_{\tau\xi}^2{\hat v}^i
+\mu J\chi^2a_{kj}D_{k\tau\xi}^3{\hat v}^iD_{\xi}(a_{lj})D_{l\tau}^2{\hat v}^i\nonumber\\[1mm]
&&+D_k(\zeta J\chi^2a_{ki})D_{\xi}(a_{lj}D_{l\tau}^2{\hat v}^i)\cdot D_{\tau\xi}^2{\hat v}^i
+\mu J\chi^2a_{ki}D_{k\tau\xi}^3{\hat v}^iD_{\xi}(a_{lj})D_{l\tau}^2{\hat v}^j]dy,\nonumber
\end{eqnarray}
and
\begin{equation}\nonumber
\begin{aligned}
L_2'=&\int_{\tilde\Omega}\kappa J\chi^2a_{kj}D_{k\tau\xi}^3\hat\theta a_{lj}D_{l\tau\xi}^3\hat\theta dy\\[1mm]
&+\int_{\tilde\Omega}\Big[D_k(\kappa J\chi^2 a_{kj})D_{\xi}(a_{lj}D_{l\tau}^2\hat\theta)
D_{\tau\xi}^2\hat\theta+\kappa J\chi^2 a_{kj}D_{k\tau\xi}^3\hat\theta D_{\xi}(a_{lj})D_{l\tau}^2\hat\theta\Big]dy\\[1mm]
&-\int_{\tilde\Omega}\Big\{D_{\tau\xi}^2(\kappa a_{kj})D_k(a_{lj}D_l\hat\theta)
+D_{\tau}(\kappa a_{kj})D_{k\xi}^2(a_{lj}D_l\hat\theta)\\[1mm]
&\quad+D_{\xi}(\kappa a_{kj})D_k\Big[D_{\tau}(a_{lj})D_l\hat\theta+a_{lj}D_{l\tau}^2\hat\theta\Big]
+\kappa a_{kj}D_{k\xi}^2D_{\tau}(a_{lj})D_l\hat\theta\Big\}\cdot J\chi^2D_{\tau\xi}^2\hat\theta dy.
\end{aligned}
\end{equation}
On the other hand, recalling that $a_{kj}D_k \hat{U}^j=0$, we have
\begin{equation}\nonumber
\begin{aligned}
L_3'=&-\frac12\int_{\tilde\Omega}J\chi^2a_{kj}D_k\hat{U}^j(|D_{\tau\xi}^2{\hat\eta}|^2+|D_{\tau\xi}^2{\hat v}^i|^2+|D_{\tau\xi}{\hat\theta}|^2)dy\\[1mm]
&-\frac12\int_{\tilde\Omega}D_k (J\chi^2a_{kj})\hat{U}^j(|D_{\tau\xi}^2{\hat\eta}|^2+|D_{\tau\xi}^2{\hat v}^i|^2+|D_{\tau\xi}^2{\hat\theta}|^2)dy\\[1mm]
=&-\frac12\int_{\tilde\Omega}D_k (J\chi^2a_{kj})\hat{U}^j(|D_{\tau\xi}^2{\hat\eta}|^2+|D_{\tau\xi}^2{\hat v}^i|^2+|D_{\tau\xi}^2{\hat\theta}|^2)dy.
\end{aligned}
\end{equation}
As for $L_4'$, in view of the following identity
\begin{eqnarray}
&&\dis{\frac{1}{\epsilon}\int_{\tilde\Omega}D_{\tau\xi}^2[a_{ki}D_k(\hat \eta
+\hat \theta)]\cdot J\chi^2D_{\tau\xi}^2{\hat v}^idy}
=-\dis{\frac{1}{\epsilon}}\int_{\tilde\Omega}J\chi^2D_{\tau\xi}^2
(a_{ki}D_k{\hat v}^i)D_{\tau\xi}^2(\hat\eta+\hat\theta)dy \nonumber \\[1mm]
&&+\dis{\frac{1}{\epsilon}}\int_{\tilde\Omega}\Big[D_{\tau\xi}^2(a_{ki})D_k(\hat\eta+\hat\theta)
+D_{\tau}(a_{ki})D_{k\xi}^2(\hat\eta+\hat\theta)
+D_{\xi}(a_{ki})D_{k\tau}^2(\hat\eta+\hat\theta)\Big]\cdot J\chi^2D_{\tau\xi}^2{\hat v}^idy  \nonumber\\[1mm]
&&+\dis{\frac{1}{\epsilon}}\int_{\tilde\Omega}D_{\xi\tau}^2(\hat\eta+\hat\theta)
[D_{\xi\tau}^2(a_{ki})D_k{\hat v}^i+D_{\tau}(a_{ki})D_{k\xi}{\hat v}^i+D_{\xi}(a_{ki})D_{k\tau}{\hat v}^i \nonumber \\[1mm]
&& -D_k(J\chi^2)a_{ki}D_{\tau\xi}^2{\hat v}^i-D_k(a_{ki})J\chi^2D_{\tau\xi}^2{\hat v}^i]dy,\nonumber
\end{eqnarray}
we deduce that
\begin{equation}\nonumber
\begin{aligned}
L_4'=&\dis{\frac{1}{\epsilon}}\int_{\tilde\Omega}\Big[D_{\tau\xi}^2(a_{ki})D_k(\hat\eta
+\hat\theta)+D_{\tau}(a_{ki})D_{k\xi}(\hat\eta+\hat\theta)
+D_{\xi}(a_{ki})D_{k\tau}^2(\hat\eta+\hat\theta)\Big]\cdot J\chi^2D_{\tau\xi}^2{\hat v}^idy\\[1mm]
&+\dis{\frac{1}{\epsilon}}\int_{\tilde\Omega}D_{\xi\tau}^2(\hat\eta +\hat\theta)[D_{\xi\tau}^2(a_{ki})D_k{\hat v}^i
+D_{\tau}(a_{ki})D_{k\xi}^2{\hat v}^i+D_{\xi}(a_{ki})D_{k\tau}^2{\hat v}^i\\[1mm]
& -D_k(J\chi^2)a_{ki}D_{\tau\xi}^2{\hat v}^i-D_k(a_{ki})J\chi^2D_{\tau\xi}^2{\hat v}^i]dy.
\end{aligned}
\end{equation}

Substituting the above estimates into \eqref{BVxxx}, using Sobolev's and Young's inequalities
and taking into account the property \eqref{xy}, we deduce that
\stepcounter{linear}
\begin{equation}\label{VVttk}
\begin{aligned}
\int_{\tilde\Omega}& \mu J\chi^2a_{kj}D_{k\tau\xi}^3{\hat v}^ia_{lj}D_{l\tau\xi}^3{\hat v}^idy
+\int_{\tilde\Omega}\zeta J\chi^2a_{ki}D_{k\tau\xi}^3{\hat v}^ia_{lj}D_{l\tau\xi}^3{\hat v}^jdy
 +\int_{\tilde\Omega} \kappa J\chi^2a_{kj}D_{k\tau\xi}^3{\hat\theta}a_{lj}D_{l\tau\xi}^3{\hat\theta}dy\\[1mm]
\le & C_9\Big[ \|U\|_3\|\eta\|_2^2+\|\tilde \bv\|_3\|\eta\|_2^2+\epsilon\|P\|_3(\|U\|_3
+\|\tilde \bv\|_3)\|\eta\|_2+\|U\|_3^2(\|\bv\|_1+\|\theta\|_1)\\[1mm]
&+\epsilon^2(\|\tilde F\|_1^2+\|\tilde G\|_1^2)+\|\tilde \bv\|_2^4+\|\eta\|_2^2(\|\tilde \bv\|_2^2
+\|\tilde\theta\|_2^2)+\|\tilde \bv\|_2^2\|\tilde\theta\|_2^2\\[1mm]
&+(\|\bv\|_1^2+\|\theta\|_1^2)\Big]+\delta\Big(\|\bv\|_3^2+\|\theta\|_3^2
+\Big\|\dis{\frac{\nabla\eta+\nabla\theta}{\epsilon}}\Big\|_1^2\Big),
\end{aligned}
\end{equation}
where $C_9$ is a constant.
\\[1mm]
{\sl Part 2. Estimate of derivatives in the normal direction}

We multiply $\eqref{LC}_2$ by $a_{3i}$ to obtain that
\stepcounter{linear}
\begin{equation}\label{LC2}
\begin{aligned}
&-(\mu+\zeta)D_3(a_{lj}D_l{\hat v}^j)+\frac{1}{\epsilon}D_3(\hat\eta+\hat\theta)\\[1mm]
=&-a_{3i}{\hat U}^ja_{kj}D_k{\hat v}^i+\epsilon a_{3i}{\hat{\tilde F}}^i
-a_{3i}{\hat{\tilde v}}^ja_{kj}D_k{\hat{\tilde v}}^i-{\hat{\tilde\theta}}
D_3\hat\eta-\hat\eta D_3\hat{\tilde\theta}\\[1mm]
&+\mu\Big[a_{3i}a_{kj}D_k(a_{lj}D_l{\hat v}^i)-D_3(a_{lj}D_l{\hat v}^j)\Big],
\end{aligned}
\end{equation}
where the last term in RHS of \eqref{LC2} can be written as follows.
\stepcounter{linear}
\begin{eqnarray}
&&\mu\Big[a_{3i}a_{kj}D_k(a_{lj}D_l{\hat v}^i)-D_3(a_{lj}D_l{\hat v}^j)\Big]\nonumber\\[1mm]
=&&\mu\Big[D_3(a_{3j})D_3{\hat v}^j+D_3(a_{\tau j})D_\tau {\hat v}^j
+a_{\tau j}D_{3\tau}^2{\hat v}^j-a_{3j}D_3(a_{3j})a_{3i}D_3{\hat v}^i\nonumber\\[1mm]
&&-a_{\tau j}a_{3i}D_\tau a_{lj}D_l{\hat v}^i-a_{\tau j}a_{\xi j}a_{3i}D_{\tau\xi}^2{\hat v}^i
-a_{3j}a_{3i}D_3(a_{\tau j})D_{\tau}{\hat v}^i\Big],\qquad \tau,\;\xi =1,2,\label{LC2l}
\end{eqnarray}
which does not include the term $D_{33}{\hv}$.

Step 1. To continue our estimate, we show the following lemma.
\begin{lemma}
There are a constant $C_{10}$ and a small $\delta >0$, such that
\stepcounter{linear}
\begin{equation}\label{VVt3d}
\begin{aligned}
&\dis{\frac{\mu+\zeta}{2}}\int_{\tilde\Omega}J\chi^2|D_{\tau3}^2(a_{lj}D_l{\hat v}^j)|^2dy
+\dis{\frac{\kappa}{2}}\int_{\tilde\Omega}J\chi^2|D_{\tau3}^2(a_{kj}D_k{\hat\theta})|^2dy\\[1mm]
\le&C_{10}\Big\{\|U\|_4^2\|\bv\|_1^2+\epsilon^2\|\tilde F\|_1^2+\|\tilde \bv\|_2^4+\|\tilde\theta\|_2^2\|\eta\|_2^2
+\int_{\tilde\Omega}J\chi^2|D_{\tau\xi y}^3\bv|^2dy\\[1mm]
&+\Big[\|U\|_3\|\eta\|_2^2+\|\tilde \bv\|_3\|\eta\|_2^2+(1+\|U\|_3)\|\theta\|_2^2
+\epsilon\|P\|_3(\|U\|_3+\|\tilde \bv\|_3)\|\eta\|_2\\[1mm]
&+\epsilon^2\|\tilde G\|_1^2+\|\tilde \bv\|_2^2\|\tilde\theta\|_2^2
+\|\tilde \bv\|_2^2\|\eta\|_2^2\Big]\Big\}+\delta(\|\bv\|_3^2+\|\theta\|_3^2).
\end{aligned}
\end{equation}
\end{lemma}
\begin{proof}
We differentiate \eqref{LC2} with respect to $y_\tau\ (\tau=1,2)$, then multiply
$-J\chi^2D_{\tau3}$ $(a_{lj}D_l{\hat v}^j)$ in $L^2(\tilde\Omega)$ to get
\stepcounter{linear}
\begin{equation}\label{*1}
\begin{aligned}
&\dis{\frac{\mu+\zeta}{2}}\int_{\tilde\Omega}J\chi^2|D_{\tau3}^2(a_{lj}D_l{\hat v}^j)|^2dy-\frac{1}{\epsilon}\int_{\tilde\Omega}J\chi^2D_{\tau3}^2(\hat\eta+\hat\theta)D_{\tau3}^2(a_{lj}D_l{\hat v}^j)dy\\[1mm]
\le&C(\|a_{3i}{\hat U}^ja_{kj}D_k{\hat v}^i\|_1^2+\epsilon^2\|\hat{\tilde F}^i\|_1^2+\|a_{3i}\hat{\tilde v}^ja_{kj}D_k\hat{\tilde v}^i\|_1^2\\[1mm]
&+\|\hat{\tilde\theta}D_3\hat\eta+\hat\eta D_3\hat{\tilde\theta}\|_1^2)+C\int_{\tilde\Omega}J\chi^2|D_{\tau\xi y}^3v|^2dy\\[1mm]
\le&C(\|U\|_2^2\|v\|_2^2+\epsilon^2\|\tilde F\|_1^2+\|\tilde v\|_2^4+\|\tilde\theta\|_2^2\|\tilde\eta\|_2^2)
+C\int_{\tilde\Omega}J\chi^2|D_{\tau\xi y}^3v|^2dy.
\end{aligned}
\end{equation}
In the mean while, we apply $D_{\tau3}$ to $\eqref{LC}_3$ and $\eqref{LC}_1$, take the product
of the resulting equations with $J\chi^2D_{\tau3}^2\hat\theta$ and $J\chi^2D_{\tau3}^2\hat\eta$
in $L^2(\tilde\Omega)$, and sum then two identities to get
\stepcounter{linear}
\begin{equation}\label{dou1}
\begin{aligned}
&-\int_{\tilde\Omega}\kappa D_{\tau3}\Big[a_{kj}D_k(a_{lj}D_l\hat\theta)\Big]
\cdot J\chi^2D_{\tau3}^2\hat\theta dy +\dis{\frac{1}{\epsilon}}
\int_{\tilde\Omega} D_{\tau3}^2(a_{kj}D_k\hat v^j)\cdot J\chi^2(D_{\tau3}^2\hat\theta+D_{\tau3}^2\hat\eta)dy\\[1mm]
&+\int_{\tilde\Omega}\big[D_{\tau3}^2(\hat U^ja_{kj}D_k\hat\eta)\cdot J\chi^2D_{\tau3}^2\hat\eta
+D_{\tau3}^2(\hat U^ja_{kj}D_k\hat\theta)\cdot J\chi^2D_{\tau3}^2\hat\theta\big] dy\\[1mm]
=& -\int_{\tilde\Omega} D_{\tau3}^2\Big\{\hat{\tilde v}^ja_{kj}D_k\hat\eta
+\hat\eta a_{kj}D_k\hat{\tilde v}^j+\epsilon a_{kj}D_k\Big[\hat P(\hat U^j
+\hat{\tilde v}^j)\Big]\Big\}\cdot J\chi^2D_{\tau3}^2\hat\eta\\[1mm]
& +\int_{\tilde\Omega} D_{\tau3}^2\Big(\epsilon\hat{\tilde G}-\hat{\tilde v}^ja_{kj}D_k\tilde\theta
-\hat\eta a_{kj}D_k\hat{\tilde v}^j -\hat{\tilde\theta} a_{kj}D_k\hat{\tilde v}^j\Big)
\cdot J\chi^2D_{\tau3}^2\hat\theta dy.
\end{aligned}
\end{equation}
We denote LHS of \eqref{dou1}$:= L_1''+L_2''+L_3''$. To control $L_k''$, we integrate by part to deduce that
\begin{equation}\nonumber
\begin{aligned}
L_1{''}=&\kappa\int_{\tilde\Omega}J\chi^2D_{\tau3}^2(a_{kj}D_k \hat\theta)D_{\tau3}^2(a_{lj}D_l \hat\theta)dy\\[1mm]
&-\int_{\tilde\Omega}\kappa J\chi^2D_{\tau3}^2\hat\theta\Big(D_{\tau3}^2(a_{kj})D_k(a_{lj}D_l\hat\theta)
+D_{\tau}(a_{kj})D_{3k}^2(a_{lj}D_l\hat\theta) +D_{3}(a_{kj})D_{k\tau}^2(a_{lj}D_l\hat\theta)\Big)dy\\[1mm]
&+\kappa\int_{\tilde\Omega}D_k(J\chi^2 a_{kj})D_{\tau3}^2(\hat\theta)D_{\tau3}^2(a_{lj}D_l \hat\theta)dy\\[1mm]
&-\kappa\int_{\tilde\Omega}J\chi^2\Big(D_{\tau3}^2(a_{kj})D_k\hat\theta+D_{\tau}(a_{kj})
D_{3k}^2\hat\theta+D_{3}(a_{kj})D_{k\tau}^2\hat\theta\Big)\cdot D_{\tau3}^2(a_{lj}D_l\hat\theta)dy
\end{aligned}
\end{equation}
and
\begin{equation}  \nonumber
\begin{aligned}
L_3{''}=&\int_{\tilde\Omega}J\chi^2D_{\tau3}^2\hat\eta\cdot\Big(D_{\tau3}^2(\hat U^ja_{kj})D_k\hat\eta
+D_{\tau}(\hat U^ja_{kj})D_{k3}^2\hat\eta +D_3(\hat U^ja_{kj})D_{k\tau}^2\hat\eta\Big)dy\\[1mm]
&-\frac12\int_{\tilde\Omega}D_k(J\chi^2a_{kj})\hat U^j|D_{\tau3}^2\hat\eta|^2
+J\chi^2a_{kj}D_k\hat U^j|D_{\tau3}^2|^2dy\\[1mm]
&+\int_{\tilde\Omega}\Big(D_{\tau3}^2(\hat U^ja_{kj})D_k\hat\theta+D_{\tau}(\hat U^ja_{kj})D_{k3}^2\hat\theta
+D_3(\hat U^ja_{kj})D_{k\tau}^2\hat\theta\Big)\cdot J\chi^2D_{\tau3}^2\hat\theta dy\\[1mm]
&-\frac12\int_{\tilde\Omega}D_k(J\chi^2a_{kj})\hat U^j|D_{\tau3}^2\hat\theta|^2
+J\chi^2a_{kj}D_k\hat U^j|D_{\tau3}^2\hat\theta|^2dy.
\end{aligned}
\end{equation}

Inserting the estimates for $L_1''$ and $L_3''$ into \eqref{dou1}, we obtain
\stepcounter{linear}
\begin{eqnarray}
&&\dis{\frac{\kappa}{2}}\int_{\tilde\Omega}J\chi^2|D_{\tau3}^2(a_{kj}D_k{\hat\theta})|^2dy
+\frac{1}{\epsilon}\int_{\tilde\Omega}J\chi^2D_{\tau3}^2(\hat\eta+\hat\theta)D_{\tau3}^2(a_{lj}D_l{\hat v}^j)dy\nonumber\\[1mm]
\le&&\|\theta\|_2^2+\delta\Big(\|D_{3k}^2(a_{lj}D_l\hat\theta)\|_0^2+\|D_{k\tau}^2(a_{lj}D_l\hat\theta)\|_0^2
+\|\theta\|_3^2\Big)\nonumber\\[1mm]
&&+C\Big[\|U\|_3\|\eta\|_2^2+\|U\|_3\|\theta\|_2^2+\|\tilde v\|_3\|\eta\|_2^2
+\epsilon\|P\|_3(\|U\|_3+\|\tilde v\|_3)\|\eta\|_2\nonumber\\[1mm]
&&+\epsilon^2\|\tilde G\|_1^2+\|\tilde v\|_2^2\|\tilde\theta\|_2^2+\|\tilde v\|_2^2\|\eta\|_2^2\Big].\label{*2}
\end{eqnarray}
Thanks to Sobolev's and Young's inequalities, we take the sum of \eqref{*1} and \eqref{*2} to deduce
the estimate \eqref{VVt3d}.
\end{proof}

Step 2. Now, it suffices to bound $\|D_{33}^2(a_{ij}D_i\hat v^j)\|_0$ in order to close
the estimate for ${\rm div}\bv$. We apply $D_3$ to \eqref{LC2} to find that
\stepcounter{linear}
\begin{eqnarray}
&&-(\mu+\zeta)D_{33}^2(a_{lj}D_l\hat{v}^j)+\dis{\frac{1}{\epsilon}}D_{33}^2(\hat\eta+\hat\theta)\nonumber\\[1mm]
=&&D_3(a_{3i})(-\hat U^ja_{kj}D_k\hat v^i+\epsilon\hat{\tilde F}^i-\hat{\tilde v}^ja_{kj}D_k\hat{\tilde v}^i)
-a_{3i}D_3(\hat U^ja_{kj}D_k\hat v^i-\epsilon\hat{\tilde F}^i+\hat{\tilde v}^ja_{kj}D_k\hat{\tilde v}^i)\nonumber\\[1mm]
&&-D_3(\hat{\tilde\theta}D_3\hat\eta+\hat\eta D_3\hat{\tilde\theta})
+O(1)(D_{33\tau}^3\hat v^j+D_{3l}^2\hat v^j+D_l\hat v^j).\label{LC22}
\end{eqnarray}
Now, multiplying the above equality \eqref{LC22} by $-J\chi^2D_{33}^2(a_{lj}D_l\hat v^j)$ in
$L^2(\tilde\Omega)$, one infers that
\stepcounter{linear}
\begin{equation}\label{LC22a}
\begin{aligned}
&\dis{\frac{\mu+\zeta}{2}}\int_{\tilde\Omega}J\chi^2|D_{33}^2(a_{lj}D_l\hat{v}^j)|^2dy
-\dis{\frac{1}{\epsilon}}\int_{\tilde\Omega}J\chi^2D_{33}^2(a_{lj}D_l\hat v^j)\cdot D_{33}^2(\hat\eta+\hat\theta)dy\\[1mm]
\le & C(\|U\|_2^2\|v\|_2^2+\epsilon^2\|\tilde F\|_1^2+\|\tilde v\|_2^4+\|\tilde\theta\|_2^2\|\tilde\eta\|_2^2)
+\|v\|_2^2+\|v\|_1^2+\int_{\tilde\Omega}J\chi^2|D_{33\tau}^3{\hv}|^2dy.
\end{aligned}
\end{equation}
Correspondingly, applying $D_{33}^2$ to \eqref{LC}$_1$ and \eqref{LC}$_3$ and multiplying
the resulting equations by $J\chi^2D_{33}^2\hat\eta$ and $J\chi^2D_{33}^2\hat\theta$ respectively, we get
\stepcounter{linear}
\begin{equation}\label{LC22b}
\begin{aligned}
&\dis{\frac{\kappa}{2}}\int_{\tilde\Omega}J\chi^2|D_{33}^2(a_{kj}D_k\hat\theta)|^2dy
+\dis{\frac{1}{\epsilon}}\int_{\tilde\Omega}J\chi^2D_{33}^2(a_{lj}D_l\hat v^j)\cdot D_{33}^2(\hat\eta+\hat\theta)dy
\\[1mm]
\le & C\|\theta\|_2^2(1+\|U\|_3)+\delta\|\theta\|_3^2
+\|U\|_3\|\eta\|_2^2+\|\tilde v\|_3\|\eta\|_2^2 +\epsilon\|P\|_3(\|U\|_3+\|\tilde v\|_3)\|\eta\|_2\\[1mm]
&+\epsilon^2\|\tilde G\|_1^2+\|\tilde v\|_2^2\|\tilde\theta\|_2^2+\|\tilde v\|_2^2\|\eta\|_2^2.
\end{aligned}
\end{equation}

Combining \eqref{LC22a} with \eqref{LC22b}, we see that there are a constant $C_{11}$ and a small
$\delta$, such that
\stepcounter{linear}
\begin{eqnarray}
&&\dis{\frac{\mu+\zeta}{2}}\int_{\tilde\Omega}J\chi^2|D_{33}^2(a_{lj}D_l\hat{v}^j)|^2dy
+\dis{\frac{\kappa}{2}}\int_{\tilde\Omega}J\chi^2|D_{33}^2(a_{kj}D_k\hat\theta)|^2dy \nonumber\\[1mm]
\le&&\delta(\|\theta\|_3^2+\|\bv\|_3^2)+C_{11}\Big\{(\|U\|_4^2\|\bv\|_1^2+\epsilon^2\|\tilde F\|_1^2
+\|\tilde \bv\|_2^4+\|\tilde\theta\|_2^2\|\tilde\eta\|_2^2) \nonumber\\[1mm]
&& +\Big[\|\bv\|_1^2+\int_{\tilde\Omega}J\chi^2|D_{33\tau}^3{\hv}|^2dy
+\|\theta\|_1^2(1+\|U\|_3) +\|U\|_3\|\eta\|_2^2+\|\tilde \bv\|_3\|\eta\|_2^2 \nonumber\\[1mm]
&& +\epsilon\|P\|_3(\|U\|_3+\|\tilde \bv\|_3)\|\eta\|_2
+\epsilon^2\|\tilde G\|_1^2+\|\tilde \bv\|_2^2\|\tilde\theta\|_2^2+\|\tilde \bv\|_2^2\|\eta\|_2^2\Big]\Big\}.\label{VV33d}
\end{eqnarray}

Step 3. To control the term $D_{33\tau}^3\hat\bv$ on RHS of \eqref{VV33d},
we introduce an auxiliary Stokes problem in the original coordinates in the region near the boundary:
\begin{equation}\nonumber
\left\{\begin{array}{r}
-\mu\triangle_x\big[(\chi D_\tau\hat\bv)\circ\Lambda^{-1}\big]
+\dis{\frac{1}{\epsilon}}\nabla_x\big[(\chi D_\tau(\hat\eta
+\hat\theta))\circ\Lambda^{-1}\big]=G_1\ \ \ \ {\rm in}\ \ W\cap\Omega,\\[1mm]
{\rm div}_x[(\chi D_\tau\hat\bv)\circ\Lambda^{-1}]=G_2\ \ \ \ {\rm in}\ \ W\cap\Omega,\\[1mm]
(\chi D_\tau\hat\bv)\circ\Lambda^{-1}=0 \ \ \ \ \ \ {\rm in}\ \ W\cap\Omega,
\end{array}\right.
\end{equation}
where
$$\begin{aligned}
G_1^i=&\chi D_\tau\Big[\zeta a_{ki}D_k(a_{lj}D_l\hat v^j)+\epsilon\hat\tilde F^i
-\hat\tilde v^j a_{kj}D_k\hat\tilde v^i-\hat\tilde\theta a_{ki}D_k\hat\eta-\hat\eta a_{ki}D_k\hat\tilde\theta\\[1mm]
&-\hat U^j a_{kj}D_k\hat v^i\Big]+o(1)\Big[D_l\hat v^i+D_{kl}^2\hat v^i+\frac{1}{\epsilon}D_k(\hat\eta+\hat\theta)\Big]
\end{aligned}$$
and
$$\begin{aligned}
G_2^i=o(1)(D_\tau\hat v^j+D_{k}\hat v^j+D_{\tau k}^2\hat v^j),
\end{aligned}$$
which can be bounded as follows.
\stepcounter{linear}
\begin{equation}\label{G1}
\begin{aligned}
\|G_1\|_{L^2}^2\le&C(\Big\|\frac{\nabla(\eta+\theta)}{\epsilon}\Big\|_{L^2}^2
+\int_{\tilde\Omega}J\chi^2|D_{\tau k}^2(a_{lj}D_l\hat v^j)|^2dy)+\delta\|\bv\|_3^2\\[1mm]
&+C_{\delta}\|\bv\|_1^2+C(\epsilon^2\|\tilde F\|_1^2+\|\tbv\|_2^4
+\|\tilde\theta\|_2^2\|\tilde\eta\|_2^2+\|U\|_3^4\|\bv\|_1^2)
\end{aligned}
\end{equation}
and
\stepcounter{linear}
\begin{equation}\label{G2}
\|G_2\|_{H^1}^2\le \delta\|\bv\|_3^2+C_{\delta}\|\bv\|_1^2
+C\int_{\tilde\Omega}J\chi^2|D_{\tau k}(a_{lj}D_l\hat v^j)|^2dy.
\end{equation}

Due to the regularity theory of the Stokes problem (see \cite{Galdi94}), one has
\stepcounter{linear}
\begin{equation}\label{stokesEb}
\int_{W\cap\Omega}|\triangle_x(\chi D_{\tau}\hv)\circ\Lambda^{-1}(x)|^2dx\le C(\|G_1\|_{L^2(W\cap\Omega)}^2
+\|G_2\|_{H^1(W\cap\Omega)}^2),
\end{equation}
where the left-hand side of \eqref{stokesEb} is equal to
$$\begin{aligned}
& \int_{W\cap\Omega}|\triangle_x(\chi D_{\tau}\hv)\circ\Lambda^{-1}(x)|^2dx
=\int_{\tilde\Omega}J\Big|\sum_{j=1}^3\sum_{k=1}^3a_{kj}D_k(\sum_{l=1}^3a_{lj}D_l(\chi D_\tau\hv))\Big|^2dy
\\[1mm]
& \qquad =\int_{\tilde\Omega}J\chi^2\Big|\sum_{j,k,l=1}^3a_{kj}a_{lj}D_{kl\tau}^3\hv\Big|^2dy
+o(1)\int_{\tilde\Omega}(|D_\tau\hv|^2+|D_{y\tau}^2\hv|^2)dy.
\end{aligned}$$
And we use \eqref{3.5.4} to get
$$D_{33\tau}^3\hv=\sum_{j,k,l=1}^3a_{kj}a_{lj}D_{kl\tau}^3\hv
-\sum_{1\le k,l\le2}\sum_{j=1}^3a_{kj}a_{lj}D_{kl\tau}^3\hv ,$$
from which, \eqref{G1} and \eqref{G2} it follows that the inequality \eqref{stokesEb} gives
\stepcounter{linear}
\begin{eqnarray} &&
(C_{11}+1)\int_{\tilde\Omega}J\chi^2|D_{33\tau}^3\hv|^2dy \nonumber \\[1mm]
&& \quad\le C(\|G_1\|_{L^2(W\cap\Omega)}^2+\|G_2\|_{H^1(W\cap\Omega)}^2)
+C\int_{\tilde\Omega}J\chi^2|D_{\tau\xi\zeta}^3\hv|^2dy+C_\delta\|\nabla\bv\|_{L^2}^2+\delta\|\bv\|_3^2\nonumber\\[1mm]
&& \quad\le \delta\|\bv\|_3^2+C_{12}\Big[\|\bv\|_1^2+\Big\|\dis{\frac{\nabla(\eta+\theta)}{\epsilon}}\Big\|_0^2
+\int_{\tilde\Omega}J\chi^2(|D_{\xi\tau y}^3\hv|^2+|D_{3\tau}^2(a_{lj}D_l\hv^j)|^2)dy\nonumber\\[1mm]
&& \qquad +(\epsilon^2\|\tilde F\|_1^2+\|\tbv\|_2^4+\|\tilde\theta\|_2^2\|\eta\|_2^2+\|U\|_3^4\|\bv\|_1^2)\Big].
\label{VV33t}   \end{eqnarray}

Now, letting
$$\Phi_\chi:=\int_{\tilde\Omega}J\chi^2\big( |D_{\tau \xi y}^3\hv|^2+|D_{\tau 3}^2(a_{lj}D_l\hat v^j)|^2
+|D_{33}^2(a_{lj}D_l\hat v^j)|^2+|D_{33\tau}^3\hv|^2\big)dy $$
and
$$\Psi_\chi:=\int_{\tilde\Omega}J\chi^2\big( a_{kj}|D_{k\tau \xi}^3\hat\theta|^2
+|D_{\tau 3}^2(a_{kj}D_k\hat\theta)|^2+|D_{33}^2(a_{kj}D_k\hat\theta)|^2\big) dy, $$
we can apply Cauchy-Schwarz's and Young's inequalities as well as the estimate \eqref{EoS0}
to deduce from \eqref{VVttk}, \eqref{VVt3d}, \eqref{VV33d} and \eqref{VV33t} that
\stepcounter{linear}
\begin{equation}\label{BE}
\begin{aligned}
\Phi_\chi+\Psi_\chi\le& \|U\|_3\|\eta\|_2^2+(\|\tbv\|_3+\|\tbv\|_2^2)\|\eta\|_2^2
+\epsilon\|P\|_3(\|U\|_3+\|\tbv\|_3)\|\eta\|_2\\[1mm]
&+C(1+\|U\|_3^4)(\|\bv\|_1^2+\|\theta\|_1^2)+\delta\Big(\|\bv\|_3^2+\|\theta\|_3^2
+\Big\|\dis{\frac{\nabla(\eta+\theta)}{\epsilon}}\Big\|_1^2\Big)\\[1mm]
&+C_\delta\Big[\epsilon^2(\|\tilde F\|_1^2+\|\tilde G\|_1^2)+\|\tbv\|_2^4
+(\|\tilde\theta\|_2^2+\|\tbv\|_2^2)\|\eta\|_2^2+\|\tbv\|_2^2\|\tilde\theta\|_2^2\Big],
\end{aligned}
\end{equation}
which, together with \eqref{EoS1}--\eqref{lmc3-i}, results in
\stepcounter{linear}
\begin{equation}\label{EoVT}
\begin{aligned}
&\|\bv\|_3^2+\|\theta\|_3^2+\Big\|\frac{\nabla\eta+\nabla\theta}{\epsilon}\Big\|_1^2\\[1mm]
\le&C_{12}\Big\{\|U\|_3\|\eta\|_2^2+(1+\|U\|_3^4)\Big[\epsilon\|P\|_3(\|U\|_3+\|\tbv\|_3)
+(\|\tilde\theta\|_2^2+\|\tbv\|_2^2+\|\tbv\|_3)\Big]\|\eta\|_2^2\\[1mm]
&+(1+\|U\|_3^4)(\|\bv\|_1^2+\|\theta\|_1^2)+(1+\|U\|_3^4)\Big[\epsilon^2(\|\tilde F\|_1^2
+\|\tilde G\|_1^2)+\|\tbv\|_2^4+\|\tbv\|_2^2\|\tilde\theta\|_2^2\Big]\Big\}.
\end{aligned}
\end{equation}

\subsubsection{Boundedness of $\eta$}

In the next lemma, we derive upper bounds of $\|\eta\|_1$ and $\|\eta\|_2$.
\begin{lemma}\label{lmc4}
There are a small $\delta>0$, and two positive constants $C_{13}$ and $C_{14}$ independent of $\epsilon$,
such that
\stepcounter{linear}
\begin{eqnarray}
\|\eta\|_1 & \le& C_{13} \Big[\epsilon^2\|\tilde F\|_0+\epsilon\big(\|\tbv\|_2\|\tbv\|_1
+\|\tilde\theta\|_2\|\eta\|_2+\|U\|_2\|\bv\|_2+\|{\rm div}\bv\|_1+(\|\tilde F\|_{-1}+\|\tilde G\|_{-1})\big)
\nonumber \\
&& +\epsilon\|P\|_2(\|U\|_2+\|\tbv\|_2)
+\|\tbv\|_3^{\frac{1}{2}}\|\eta\|_0+\|\tbv\|_2^2+\|\eta\|_1\|\tilde\theta\|_1
+\|\tbv\|_1(\|\tilde\theta\|_1+\|\eta\|_1)\Big] \label{ETx}
\end{eqnarray}
and
\stepcounter{linear}
\begin{eqnarray}
\|\eta\|_2 &\le& C_{14}(1+\epsilon)(1+\|U\|_2^2)\Big\{\epsilon(\|\tilde F\|_1+\|\tilde G\|_{-1})
+\|\tilde \bv\|_3^2+\|\tilde\theta\|_3\|\eta\|_2 +\|\tilde\bv\|_3^{\frac12}\|\eta\|_0 \nonumber \\
&& +\Big[\epsilon\|P\|_2(\|U\|_2+\|\tilde\bv\|_2)\|\eta\|_0\Big]^{\frac12}
+\|\tbv\|_1(\|\tilde\theta\|_1+\|\eta\|_1)\Big\} +\epsilon\|\bv\|_{3}+\delta\|\theta\|_3.
 \label{ETxx}  \end{eqnarray}
\end{lemma}
\begin{proof}
From \eqref{EoS0} we get
\begin{equation}\nonumber
\begin{aligned}
\epsilon\|\bv\|_{2}+\|\nabla\eta+\nabla\theta\|_0
\le C\big[\epsilon^2\|\tilde F\|_0+\epsilon(\|\tbv\|_2\|\tbv\|_1
+\|\tilde\theta\|_2\|\eta\|_2+\|U\|_2\|\bv\|_2+\|{\rm div}\bv\|_1)\big],
\end{aligned}
\end{equation}
which together with Lemma \ref{lmc1} gives
\begin{eqnarray*}
&&\|\nabla\eta\|_1\le\|\nabla\eta+\nabla\theta\|_0+\|\nabla\theta\|_0  \\[1mm]
\le&& C\big[\epsilon^2\|\tilde F\|_0+\epsilon(\|\tbv\|_2\|\tbv\|_1
+\|\tilde\theta\|_2\|\eta\|_2+\|U\|_2\|\bv\|_2+\|{\rm div}\bv\|_1)\big] \\[1mm]
&&+C\Big[\|\tbv\|_3\|\eta\|_0^2+\epsilon^2(\|\tilde F\|_{-1}^2+\|\tilde G\|_{-1}^2)
+\epsilon\|P\|_2(\|U\|_2+\|\tbv\|_2)\|\eta\|_0\\[1mm]
&&+\|\tbv\|_1^4+\|\eta\|_1^2\|\tilde\theta\|_1^2
+\|\tbv\|_1^2(\|\tilde\theta\|_1^2+\|\eta\|_1^2)\Big]^{1/2}.
\end{eqnarray*}
If we apply Poincar\'e's and Young's inequalities to the above inequality, and use the fact that
\stepcounter{linear}
\begin{equation}\label{fact}
(A_1+A_2+\cdots+A_n)^{1/2}\le A_1^{1/2}+A_2^{1/2} +\cdots +A_n^{1/2}\quad
\mbox{for } A_i\ge0\;\;\; (i=1,\cdots,n),
\end{equation}
we obtain the estimate \eqref{ETx} immediately.

On the other hand, from the estimate \eqref{EoS1} we conclude that
\stepcounter{linear}
\begin{equation}\nonumber
\begin{aligned}
\|\nabla\eta\|_1 & \le \|\nabla\eta+\nabla\theta\|_1+\|\nabla\theta\|_1 \\[1mm]
& \le C\epsilon(1+\|U\|_2^2)\Big\{\epsilon(\|\tilde F\|_1+\|\tilde G\|_{-1})
+\|\tilde \bv\|_3^2+\|\tilde\theta\|_3\|\eta\|_2 +\|\tilde\bv\|_3^{\frac12}\|\eta\|_0 \\[1mm]
& \quad +\Big[\epsilon\|P\|_2(\|U\|_2+\|\tilde\bv\|_2)\|\eta\|_0\Big]^{\frac12}
+\|\tbv\|_1(\|\tilde\theta\|_1+\|\eta\|_1)\Big\} +\epsilon\|\bv\|_{3}
 +\delta\|\theta\|_3+C_\delta\|\theta\|_1,
\end{aligned}
\end{equation}
which, together with Poincar\'e's inequality, \eqref{ETx} and \eqref{fact}, implies \eqref{ETxx}.
\end{proof}

\section{Existence of the nonlinear problem}
\newcounter{exist}
\renewcommand{\theequation}{\thesection.\arabic{equation}}
\setcounter{equation}{0}

In this section, we give the proof of the existence for the
nonlinear problem (\ref{1}) by using the Tikhonov theorem which can be
found in \cite{NS04}. For completeness, we state the theorem in the
following.
\begin{theorem} (Tikhonov Theorem, \cite[P72, 1.2.11.6]{NS04}) \ Let $M$ be a nonempty
bounded closed convex subset of a separable reflexive Banach space
$X$ and let $F : M \rightarrow M$ be a weakly continuous mapping
(i.e., if $x_n \in M$, $x_n \rightharpoonup x$ weakly in $X$, then
$F(x_n) \rightharpoonup F(x)$ weakly in $X$ as well). Then $F$ has at least one fixed point in $M$.
\end{theorem}

Define a Banach space $X$ by
$$X=\bar{H}^1\times {H}_0^1\times{H}_0^1 ,$$
which can be easily verified to be separable and reflexive.

A convex subset $K_1(E)$ of $X$ is defined by
$$K_1(E)=\big\{(\bv, \theta)\in (H^3\cap H_0^1))\times (H^3\cap H_0^1)\;\, \big|\;\, \|\bv\|_3+\|\theta\|_3\le E\big\},$$
where $E<1$ is a small positive constant. By the lower semi-continuity of norms, we easily see that
the subset $K_1(E)$ is also closed in $X$.

We define a space $K$ by
$$K=K_0\times K_1(E),$$
where $K_0$ is defined by (\ref{k0}). Note that $K$ is a nonempty bounded closed convex subset of $X$.

Now, we define a nonlinear operator $N$ from $K$ to $X$ by
$$N(\tilde U, \tilde \bv, \tilde\theta):=(U, \bv,\theta),$$
where $U$ and $(\bv,\theta)$ are the solutions of \eqref{i1l}
and \eqref{CL} for given $(\tilde U,\tilde \bv,\tilde\theta)$, respectively.

Next, we want to find a fixed point $(U,\bv,\theta)$ of $N$ in $K$, such that $(U, \bv,\theta)=N(U, \bv, \theta)$,
which, together with the existence of weak solutions in Lemmas \ref{lmI1} and \ref{lmc1}, gives that
$(U,P)$ and $(\eta,\bv,\theta)$ are solutions of the boundary value problems \eqref{i1} and \eqref{c1},
respectively. So $(U+\bv, \epsilon P+\eta, \theta)$ will be a solution to \eqref{1}.
For this purpose, we have to show that $N$ maps $K$ into itself
and $N:K\rightarrow K$ is a weakly continuous mapping.
\begin{lemma}\label{lmex1}
There is a small constant $\epsilon_0>0$, depending only on $\Omega ,\mu,\lambda$, $\f$ and $\g$,
such that for any $\epsilon\in (0,\epsilon_0)$,
$K$ is a nonempty bounded closed convex subset of $X$ and $N(K)\subset K$.
\end{lemma}
\begin{proof}
By virtue of the definition, it is obvious that $K\subset X$ is a nonempty, bounded, closed convex set.
Now, we will show that
the operator $N$ maps $K$ into itself,
i.e., $N(K)\subset K$. To this end, let $(\tilde U, \tilde v, \tilde\theta)\subset K$
and $(U, \bv,\theta)=N(\tilde U, \tilde \bv, \tilde\theta)$. By Lemmas \ref{lmI1} and \ref{lmI2},
we see that $U\in K_0$ for all $\tilde U\in K_0$. Thus, it suffices to check that
$(\bv,\theta)\in K_1(E)$ for $(\tilde \bv, \tilde\theta)\in K_1(E)$.
By \eqref{lml2-2} and \eqref{lml2-3}, we have
\stepcounter{exist}
\begin{equation}\label{ex1}
\|U\|_3+\|\nabla P\|_1\le M_1\quad\mbox{ and }\quad \|U\|_4+\|\nabla P\|_2\le M_2,
\end{equation}
where $M_1=C_3\|\h\|_1(\|\h\|_1+1)^8$ and $M_2=C_4\|\h\|_2(\|\h\|_2+1)^{12}$.

On the other hand, recalling the definition of $\tilde F$ and $\tilde G$, we get from \eqref{ex1} that
\stepcounter{exist}
\begin{equation}\label{F}
\begin{aligned}
\|\tilde F\|_1 =&\|(\epsilon P+\eta)\f -(\epsilon P+\eta)(U+\tilde \bv)
\cdot\nabla (U+\tilde \bv)-\tilde \theta\nabla P- P\nabla\tilde \theta\|_1  \\[1mm]
\leq & C\big[\|\eta\|_2(\|\f\|_1+\|U\|_2^2+\|\tbv\|_2^2)+\epsilon\|P\|_2(\|\f\|_1
+\|U\|_2^2+\|\tbv\|_2^2+\|\tilde\theta\|_2)\big]  \\[1mm]
\leq & C\|\eta\|_2(\|\f\|_1+(M_1+1)^2)+\epsilon C M_1(\|\f\|_1+(M_1+1)^2),
\end{aligned}
\end{equation}
\stepcounter{exist}
\begin{equation}\label{G}
\begin{aligned}
\|\tilde G\|_1=&\|\tilde \Psi-(\epsilon P+\eta)(U+\tilde \bv)\cdot\nabla\tilde \theta
+(\epsilon P+\eta)\tilde\theta{\rm div}\tbv+P{\rm div}\tbv\|_1\\[1mm]
\leq & C\big[\epsilon(\|U\|_2^2+\|\tbv\|_2^2)+(\|\eta\|_2+\epsilon\|P\|_2)
(\|U\|_2+\|\tbv\|_2)\|\tilde\theta\|_2+\|P\|_2\|\tbv\|_2\big]\\[1mm]
\leq & C\|\eta\|_2(M_1+1)+C(\epsilon(M_1+1)^2+(M_1+1)).
\end{aligned}
\end{equation}
As a result of Poincar\'e's inequality and Lemma \ref{lmc1}, we have
\stepcounter{exist}
\begin{eqnarray}
\|\bv\|_1+\|\theta\|_1
& \leq & C\Big\{(E^\frac12+E)\|\eta\|_2+\epsilon\Big[\|\eta\|_2(M_1+1)
+C\big(\|\eta\|_2(\|\f\|_1+(M_1+1)^2\big)   \nonumber \\[1mm]
&& +\epsilon C M_1(\|\f\|_1+(M_1+1)^2)+\epsilon(M_1+1)^2+(M_1+1))\Big]   \nonumber \\[1mm]
&& +\epsilon C_\delta(M_1+1)^2+E^2\Big\}+\delta\|\eta\|_2    \nonumber \\[1mm]
& \le & C\big[E^\frac12+\epsilon(\|\f\|_1+(M_1+1)^2)\big]\|\eta\|_2
+\epsilon^2 C M_1(\|\f\|_1+(M_1+1)^2)+CE^2   \nonumber  \\[1mm]
&& +\delta\|\eta\|_2+\epsilon C_\delta(M_1+1)^2,  \label{Vx}
\end{eqnarray}
where we have used the estimates \eqref{ex1}, \eqref{F} and \eqref{G}.

On the other hand, in view of Poincar\'e's and Young's inequalities, \eqref{Vx} and \eqref{ETxx}
in Lemma \ref{lmc4}, we find that
\stepcounter{exist}
\begin{eqnarray}
\|\eta\|_2
& \leq & C(1+\epsilon)(1+M_1^2)\Big\{\epsilon\Big[\|\eta\|_2(\|\f\|_1+(M_1+1)^2)+ M_1(\|\f\|_1+(M_1+1)^2)\nonumber\\[1mm]
&&+\|\eta\|_2(M_1+1)+(\epsilon(M_1+1)^2+(M_1+1))\Big]+E^2+(E+E^{\frac12})\|\eta\|_2\nonumber\\[1mm]
&& +\epsilon^{\frac12}M_1^{\frac12}(M_1+E)^{\frac12}\|\eta\|_2^{\frac12}+E(E+\|\eta\|_2)\Big\}
+\epsilon\|\bv\|_3+\delta\|\theta\|_3\nonumber\\[1mm]
&\leq & C_{15}(1+\epsilon)(1+M_1^2)\Big\{ \epsilon\Big[\|\f\|_1+(M_1+1)^2+(M_1+1)\Big]+(E+E^{\frac12})+\delta
\Big\}\|\eta\|_2\nonumber\\[1mm]
&&  +\epsilon\|\bv\|_3+\delta\|\theta\|_3+C_{15}(1+\epsilon)(1+M_1^2)\Big\{\epsilon (M_1+1)^2\nonumber \\[1mm]
&& +\epsilon\Big[ M_1(\|\f\|_1+(M_1+1)^2)+(\epsilon(M_1+1)^2+(M_1+1))\Big]+E^2\Big\},
\label{ET}
\end{eqnarray}
where $C_{15}$ is a positive constant.

Combining \eqref{EoVT} with \eqref{F}--\eqref{ET}, we conclude that
there is a constant $C_{16}$, such that
\stepcounter{exist}
\begin{eqnarray}
\|\bv\|_3 +\|\theta\|_3+\|\eta\|_2
\leq && C_{16}(1+M_1)^5\Big\{ \epsilon\Big[\|\f\|_1+(M_1+1)^2\Big]+(E+E^{\frac12})+\delta
\Big\}\|\eta\|_2 \nonumber\\[1mm]
&& +\epsilon M_1^{\frac12}\|\bv\|_3+\delta M_1^{\frac12}\|\theta\|_3+C_{16}(1+M_1)^5\Big\{\epsilon(M_1+1)^2 \nonumber\\[1mm]
&& +\epsilon(1+ M_1)\Big[\|\f\|_1+(M_1+1)^2\Big]+E^2\Big\}.\label{VTE}
\end{eqnarray}

Thus, first taking $\delta$ small enough and then choosing $\epsilon_0$ and $E$ suitably small, such that
$$\begin{aligned}
\ &C_{16}(1+M_1)^5\Big\{ \epsilon_0\Big[\|\f\|_1+(M_1+1)^2\Big]+(E+E^{\frac12})\Big\}<1,\quad
 \; \epsilon_0 M_1^{\frac12}<1, \\[1mm]
\ & C_{16}(1+M_1)^5\Big\{\epsilon_0(M_1+1)^2+\epsilon_0(1+ M_1)\Big[\|\f\|_1+(M_1+1)^2\Big]+E^2\Big\}<E,
\end{aligned}$$
we deduce from (\ref{VTE}) that for all $\epsilon\in (0,\epsilon_0)$,
\begin{equation}\nonumber
\begin{aligned}
\|\bv\|_3+\|\theta\|_3+\|\eta\|_2\le E,
\end{aligned}
\end{equation}
which gives $\|\bv\|_3+\|\theta\|_3\le E$ immediately. This completes the proof.
\end{proof}
\begin{lemma}\label{lmex2}
Let $N,X,K_0$ and $K_1(E)$ be the same as in Lemma \ref{lmex1}. Then $N:K\rightarrow K$
is a weakly continuous mapping.
\end{lemma}
\begin{proof}
By the definition of weakly continuous mapping (see, for example, \cite[P72,1.4.11.6]{NS04}),
it suffices to prove that $N$ is continuous on $K$ in the norm of $X$.

Let $(U_i,\bv_i,\theta_i)=N(\tilde U_i, \tilde \bv_i, \tilde\theta_i)$, $i=1,2$.
In particular, let $(U_i,P_i)\in(H^4\cap H_{0,\sigma}^1)\times \bar{H}^3$ and
$(\eta_i,\bv_i,\theta_i)\in\bar{H}^2\times (H^3\cap H_0^1)\times (H^3\cap H_0^1)$ be the solutions
of \eqref{i1l} and \eqref{CL} for given $(\tilde U_i, \tilde \bv_i, \tilde\theta_i)$ respectively,
i.e.,
\stepcounter{exist}
\begin{equation}\label{IL1}
\left\{\begin{array}{llll}
(\tilde U_i+\tilde \bv_i)\cdot\nabla U_i-\mu\triangle U_i+\nabla P_i=\h,\quad \dis{\int}P_idx=0, \\[1mm]
{\rm div}\, U_i=0;
\end{array}\right.
\end{equation}
and
\stepcounter{exist}
\begin{equation}
\label{CL1}
\left\{\begin{array}{llll}
U_i\cdot\nabla \eta_i + \dis{\frac{{\rm div}\bv_i}{\epsilon}}
=-\tilde \bv_i\cdot\nabla\eta_i-\eta_i{\rm div}\tbv_i-\epsilon{\rm div}(P_i(U_i+\tilde \bv_i)),\\[2mm]
U_i\cdot\nabla \bv_i-\mu\triangle \bv_i-\zeta\nabla{\rm div}\bv_i
+\dis{\frac{\nabla \eta_i+\nabla\theta_i}{\epsilon} } =\epsilon \tilde F_i
-\tilde \bv_i\cdot\nabla \tilde \bv_i-\tilde \theta_i\nabla\eta_i-\eta_i\nabla\tilde \theta_i,\\[2mm]
U_i\cdot\nabla\theta_i-\kappa\triangle\theta_i +\dis{\frac{{\rm div}\bv_i}{\epsilon} }
=\epsilon \tilde G_i-\tilde \bv_i\cdot\nabla\tilde \theta_i
-\eta_i{\rm div}\tbv_i-\tilde \theta_i{\rm div}\tbv_i,
\end{array}\right.
\end{equation}
where the force $\tilde F_i$ and heat source $\tilde G_i$ are given by
\begin{eqnarray*} &&
\tilde F_i=(\epsilon P_i +\eta_i)\f -(\epsilon P_i+\eta_i)(U_i+\tilde \bv_i)
\cdot\nabla(U_i+\tilde \bv_i) -\tilde \theta_i\nabla P_i- P_i\nabla\tilde \theta_i, \\[1mm]
&& \tilde G_i=\tilde\Psi_i -(\epsilon P_i+\eta_i)(U_i+\tilde \bv_i)\cdot\nabla\tilde\theta_i
+(\epsilon P_i+\eta_i)\tilde\theta_i{\rm div}\tbv_i+P_i{\rm div}\tbv_i.
\end{eqnarray*}

Now, if we set
$$ \begin{aligned}
&W=U_2-U_1,\ \ \tilde W=\tilde U_2-\tilde U_1,\ \ Q=P_2-P_1,\ \ \xi=\eta_2-\eta_1, \\[1mm]
&\bw=\bv_2-\bv_1,\ \ \   \tilde \bw=\tilde \bv_2-\tilde \bv_1,\ \ \  \beta=\theta_2-\theta_1,
\ \ \ \tilde \beta=\tilde \theta_2-\tilde \theta_1,\\[1mm]
&J=\tilde F_2-\tilde F_1,\ \ \ I=\tilde G_2-\tilde G_1 ,
\end{aligned}$$
then, we can have the following systems:
\stepcounter{exist}
\begin{equation}\label{IL12}
\left\{\begin{array}{llll}
(\tilde U_1+\tilde v_1)\cdot\nabla W-\mu\triangle W+\nabla Q
=-(\tilde W+\tbw)\cdot\nabla U_2,\quad\; \dis{\int} Qdx=0,\\[1mm]
{\rm div}W=0 ,
\end{array}\right.
\end{equation}
and
\stepcounter{exist}
\begin{equation}
\label{CL12}
\left\{\begin{array}{llll}
U_1\cdot\nabla \xi +\dis{\frac{{\rm div}\bw}{\epsilon} }
=-{\rm div}(\tilde \bv_1\xi+\tbw\eta_2)-\epsilon{\rm div}(P_1(W+\tbw)
+Q(U_2+\tilde \bv_2))-W\cdot\nabla \eta_2,\\[2mm]
U_1\cdot\nabla \bw-\mu\triangle \bw-\zeta\nabla{\rm div}\bw
+ \dis{\frac{\nabla \xi+\nabla\beta}{\epsilon} } \\[1mm]
\ \ \ \ \ =\epsilon J-\tbw\cdot\nabla \tilde \bv_2
-\tilde \bv_1\cdot\nabla\tbw-W\cdot\nabla \bv_2-\nabla(\tilde\theta_1\xi+\tilde\beta\eta_2),\\[2mm]
U_1\cdot\nabla\beta-\kappa\triangle\beta+ \dis{\frac{{\rm div}\bw}{\epsilon} }\\[1mm]
\ \ \ \ \ =\epsilon I-\tbw\cdot\nabla\tilde \theta_2-\tilde \bv_1\cdot\nabla\tilde\beta
-W\cdot\nabla\theta_2-\xi{\rm div}\tbv_1-\eta_2{\rm div}\tbw
-\tilde \beta{\rm div}\tbv_1-\tilde\theta_2{\rm div}\tbw,
\end{array}\right.
\end{equation}
where $J$ and $I$ read as
\begin{eqnarray*}
J &=& (\epsilon Q+\xi)\f -(\epsilon Q+\xi)(U_2+\tilde \bv_2)\cdot\nabla(U_2+\tilde \bv_2) \\[1mm]
&& -(\epsilon P_1+\eta_1)((W+\tbw)\cdot\nabla(U_2+\tilde \bv_2)+(U_1+\tilde \bv_1)\cdot\nabla(W+\tbw))
-\nabla(\tilde \beta P_1+ Q\tilde\theta_2), \\[1mm]
I &=& 2\mu D(W+\tbw):D(U_2+\tilde \bv_2)+2\mu D(U_1+\tilde \bv_1):D(W+\tbw)\\[1mm]
&& +\lambda {\rm div}(W+\tbw)\cdot{\rm div}(U_2+\tilde \bv_2)
+\lambda {\rm div}(U_1+\tilde \bv_1)\cdot{\rm div}(W+\tbw) \\[1mm]
&& -(\epsilon Q+\xi)(U_2+\tilde \bv_2)\cdot\nabla\tilde \theta_2 -(\epsilon P_1+\eta_1)
((W+\tbw)\cdot\nabla\tilde \theta_2+(U_1+\tilde \bv_1)\cdot\nabla\tilde \beta) \\[1mm]
&& +(\epsilon Q+\xi)\tilde\theta_2{\rm div}\tbv_2-(\epsilon P_1+\eta_1)
(\tilde\beta{\rm div}\tbv_2+\tilde\theta_1{\rm div}\tbw)
+Q{\rm div}\tbv_2+P_1{\rm div}\tbw.
\end{eqnarray*}

Note that $J$ and $I$ can be bounded as follows.
\stepcounter{exist}
\begin{equation}
\label{J}
\begin{aligned}
\|J\|_0\le& (\epsilon\|Q\|_1+\|\xi\|_1)(\|\f\|_2+\|U_2\|_2^2+\|\tilde \bv_2\|_2^2)
+(\epsilon\|P_1\|_2+\|\eta_1\|_2)(\|W\|_1+\|\tbw\|_1)\\[1mm]
&\cdot(\|U_2\|_2+\|\tilde \bv_2\|_2+\|U_1\|_2+\|\tilde \bv_1\|_2)+\|\tilde\theta\|_2\|Q\|_1+\|\tilde\beta\|_1\|P_1\|_2
\end{aligned}
\end{equation}
and
\stepcounter{exist}
\begin{equation}
\label{I}
\begin{aligned}
\|I\|_0\le& C\{(\|W\|_1+\|\tbw\|_1)(\|U_1\|_2+\|U_2\|_2+\|\tilde \bv_1\|_2
+\|\tilde \bv_2\|_2)+(\epsilon\|Q\|_1+\|\xi\|_1)(\|U_2\|_2+\|\tilde \bv_2\|_2)\\[1mm]
&\cdot\|\tilde\theta_2\|_2+(\epsilon\|P_1\|_2+\|\eta_1\|_2)
[(\|W\|_1+\|\tbw\|_1)\|\tilde\theta_2\|_2+\|U_1\|_2+\|\tilde \bv_1\|_2\|\beta\|_1]\\[1mm]
&+(\epsilon\|P_1\|_2+\|\eta_1\|_2)(\|\tilde\beta\|_1\|\tilde \bv_2\|_2
+\|\tilde\theta_1\|_2\|\tbw\|_1)+\|Q\|_1\|\tilde \bv_2\|_2+\|P_1\|_2\|\tbw\|_1\}
\end{aligned}
\end{equation}

On the one hand, we multiply $\eqref{IL12}_1$ by $W$ and make use of Poincar\'e's inequality to
deduce that
$$
\begin{aligned}
\Big(\mu-\frac{C\|\tbv\|_3}{2}\Big)\|\nabla W\|_0^2\le C(\|\tilde W\|_1+\|\tbw\|_1)\|U_2\|_3\|W\|_1,
\end{aligned}$$
where $C$ is a positive constant depending only on $\Omega$ and $\mu$. Consequently,
\stepcounter{exist}
\begin{equation}
\label{W}
\|W\|_1\le C(\|\tilde W\|_1+\|\tbw\|_1)
\end{equation}
for some positive constant $C$ depending only on $\Omega ,\mu ,\lambda ,\f ,E$ and $\epsilon_0$.

By the classical estimates for the Stokes equations
\begin{equation}\nonumber
\left\{
\begin{array}{l}
-\mu\triangle W+\nabla Q=-(\tilde W+\tbw)\cdot\nabla U_2-(\tilde U_1+\tilde \bv_1)\cdot\nabla W,\\[1mm]
{\rm div}W=0,
\end{array}\right.
\end{equation}
we obtain that
\stepcounter{exist}
\begin{equation}
\label{Q}
\begin{aligned}
\|W\|_2+\|\nabla Q\|_0\le& C (\|\tilde W+\tbw\|_1\|U_2\|_2+\|\tilde U_1+\tilde \bv_1\|_2\|W\|_1) \\[1mm]
\le &C(\|\tilde W\|_1+\|\tbw\|_1),
\end{aligned}
\end{equation}
where the estimate \eqref{W} has been used.

On the other hand, if we multiply $\eqref{CL12}_1$, $\eqref{CL12}_2$ and $\eqref{CL12}_3$ by $\xi$, $w$
and $\beta$ in $L^2$ respectively, we find that
\begin{eqnarray} &&
\mu\|\nabla \bw\|_0^2+\zeta\|{\rm div}\bw\|_0^2+\kappa\|\nabla\beta\|_0^2 \nonumber \\[1mm]
&& =-\int\Big[\xi{\rm div}\tbv_1+\eta_2{\rm div}\tbw+\tilde \bv_1\cdot\nabla\xi+\tbw\cdot\nabla\eta_2
+\epsilon{\rm div}(P_1(W+\tbw)+Q(U_2+\tilde \bv_2))+W\cdot\nabla\eta_2\Big]\xi dx  \nonumber \\[1mm]
&& \quad +\int\Big[\epsilon J-\tbw\cdot\nabla\tilde \bv_2-\tilde \bv_1\cdot\nabla\tbw
-W\cdot\nabla \bv_2-\nabla(\tilde\theta_1\xi+\tilde\beta\eta_2)\Big]\cdot w dx  \nonumber \\[1mm]
&& \quad +\int\Big[\epsilon D-\tbw\cdot\nabla\tilde\theta_2-\tilde \bv_1\cdot\nabla\tilde\beta
-W\cdot\nabla\theta_2-\xi{\rm div}\tbv_1-\eta_2{\rm div}\tbw
-\tilde\beta{\rm div}\tbv_1-\tilde\theta_1{\rm div}\tbw\Big]\beta dx  \nonumber  \\[1mm]
&& \leq C\Big\{\|\xi\|_0^2\|\tilde v_1\|_3+\|\xi\|_1\|\tbw\|_1\|\eta_2\|_1
+\epsilon\|\xi\|_0\big[\|P_1\|_2(\|W\|_1+\|\tbw\|_1)+\|Q\|_1(\|U_2\|_2+\|\tilde \bv_2\|_2)\big]  \nonumber \\[1mm]
&& \quad +\|W\|_1\|\eta_2\|_3\|\xi\|_1\Big\} +C\Big\{\epsilon \|J\|_0^2+\|\tbw\|_1^2(\|\tilde \bv_2\|_2^2
+\|\tilde \bv_1\|_2^2)+\|W\|_1^2\|\bv_2\|_2^2+\|\tilde\theta_2\|_1^2\|\xi\|_1^2   \nonumber  \\[1mm]
&& \quad +\|\tilde\beta\|_1^2\|\eta\|_2^2\Big\}+C\Big\{\epsilon\|I\|_0^2+\|\tbw\|_1^2
\|\tilde \theta_2\|_2^2+\|\tilde \bv_1\|_2^2\|\tilde\beta\|_1^2
+\|W\|_1^2\|\theta_2\|_2^2+\|\xi\|_1^2\|\tilde \bv_1\|_2^2   \nonumber \\[1mm]
&& \quad +\|\eta_2\|_2^2\|\tbw\|_1^2+\|\tilde\beta\|_1^2\|\tilde \bv_1\|_2^2
+\|\tilde\theta_1\|_2^2\|\tbw\|_1^2\Big\}+\delta(\|\bw\|_0^2+\|\beta\|_0^2).  \label{wbe}
\end{eqnarray}
Also, from the Stokes equations
\begin{equation}\nonumber
\left\{\begin{array}{l}
-\mu\epsilon\triangle \bw+\nabla\xi=\epsilon\big( \epsilon J-\tbw\cdot\nabla \tilde v_2
-\tilde \bv_1\cdot\nabla\tbw-W\cdot\nabla v_2-\nabla(\tilde\theta_1\xi+\tilde\beta\eta_2)
+\zeta\nabla{\rm div}\bw\big) -\nabla\beta,    \\[1mm]
{\rm div}\bw={\rm div}\bw,
\end{array}\right.
\end{equation}
we get the following estimate
\stepcounter{exist}
\begin{equation}
\label{xi}
\begin{aligned}
\epsilon\|\bw\|_2+\|\nabla\xi\|_0\le &\epsilon C\Big( \|{\rm div}\bw\|_2+\epsilon\|J\|_0
+\|\tbw\|_1(\|\tilde \bv_2\|_2+\|\tilde \bv_1\|_2)+\|W\|_1\|\bv_2\|_2  \\[1mm]
&+\|\tilde\theta_1\|_2\|\xi\|_1+\|\tilde\beta\|_1\|\eta_2\|_2\Big) +C\|\beta\|_1.
\end{aligned}
\end{equation}

Applying Poincar\'e's inequality and substituting \eqref{xi} into \eqref{wbe},
employing the estimates \eqref{J}--\eqref{Q} and recalling the smallness
of $\epsilon_0$ and $E$, we conclude that
$$\|W\|_1+\|\bw\|_1+\|\beta\|_1\le C(\|\tilde W\|_1+\|\tbw\|_1+\|\tilde \beta\|_1),$$
where $C$ is a positive constant depending only on $\Omega ,\mu ,\lambda,\f ,E$
and $\epsilon_0$. This completes the proof.
\end{proof}

Finally, having had Lemmas \ref{lmex1} and \ref{lmex2}, we can apply the Tikhonov fixed
point theorem to find a fixed point $(U,\bv,\theta)=N(U,\bv,\theta)$ in the set $K$. Moreover,
the pressure $P\in\bar{H}^2$ satisfies
$$\nabla P=\f+\g+\mu\triangle U-(U+\bv)\cdot\nabla U, $$
and $(U+\bv,\epsilon P+\eta,\theta)$ is a solution to \eqref{1}. Thus,
we have shown the following proposition.
\begin{prop}\label{prop5.1}
Let $\f, \g\in H^2(\Omega)$. Then, there exists an $\epsilon_0$ depending only on $\Omega ,\mu ,\lambda ,\f$
and $\g$, such that for all $\epsilon\in (0,\epsilon_0)$, there is a solution
$(U,P,\bv,\eta ,\theta)\in (H^4\cap H_{0,\sigma}^1)\times\bar{H}^3\times(H^3\cap H_0^1)\times\bar{H}^2\times(H^3\cap H_0^1)$ of
\eqref{i1} and \eqref{c1}, satisfying
$$\|\bv\|_3+\|\eta\|_2+\|\theta\|_3\le E ,$$
where $E$ is a small positive constant depending only on $\Omega ,\mu ,\lambda ,\f$ and $\g$.
Moreover, $(U+\bv,\epsilon P+\eta,\theta)$ is a solution of
the system \eqref{1} for any $\epsilon\in (0,\epsilon_0)$.
\end{prop}

\section{Incompressible limit}
\newcounter{limit}
\renewcommand{\theequation}{\thesection.\thelimit}

Let $\epsilon<\epsilon_0$ and $(U^\epsilon,\bv^\epsilon,\theta^\epsilon)\in K$ be the solution
established in Proposition \ref{prop5.1}.
We take $\bv=\tilde \bv=\bv^\epsilon$, $\theta=\tilde
\theta=\theta^\epsilon$ and $\eta=\eta^\epsilon$ in \eqref{VTE} to get that
\begin{equation}\nonumber
\begin{aligned}
\|\bv^\epsilon\|_3+\|\theta^\epsilon\|_3+\|\eta^\epsilon\|_2
\leq & C_{16}(1+M_1)^5\Big\{ \epsilon\Big[\|\f\|_1+(M_1+1)^2\Big]+(E+E^{\frac12})+\delta
\Big\}\|\eta\|_2 \\[1mm]
& +(\epsilon M_1^{\frac12}+E)\|\bv\|_3+(\delta M_1^{\frac12}+E)\|\theta\|_3 \\[1mm]
& +C_{16}(1+M_1)^5\Big\{\epsilon(M_1+1)^2+\epsilon(1+ M_1)\Big[\|\f\|_1+(M_1+1)^2\Big]\Big\}.
\end{aligned}
\end{equation}
Thus, by taking $\epsilon_0$ and $E$ so small that
$$C_{16}(1+M_1)^5\Big\{ \epsilon\Big[\|\f\|_1+(M_1+1)^2\Big]+(E+E^{\frac12})\Big\}<1, \;\;
\epsilon M_1^{\frac12}+E<1,$$
we obtain
$$ \|\bv^\epsilon\|_3+\|\theta^\epsilon\|_3+\|\eta^\epsilon\|_2
\leq C_{16}(1+M_1)^5\Big\{\epsilon(M_1+1)^2+\epsilon(1+ M_1)\Big[\|\f\|_1+(M_1+1)^2\Big]\Big\},
$$
whence,
\stepcounter{limit}
\begin{equation}\label{lim1}
\|\bv^\epsilon\|_3+\|\theta^\epsilon\|_3+\|\eta^\epsilon\|_2\to 0,\ \ \ \ {\rm as}\ \ \epsilon\rightarrow0.
\end{equation}

Furthermore, from $\eqref{c1}_1$, i.e.,
\begin{equation}\nonumber
U^\epsilon\cdot\nabla \eta^\epsilon+\frac{{\rm div}\bv^\epsilon}{\epsilon}
=-\bv^\epsilon\cdot\nabla\eta^\epsilon-\eta^\epsilon{\rm div}\bv^\epsilon
-\epsilon{\rm div}\big( P^\epsilon(U^\epsilon+\bv^\epsilon)\big)
\end{equation}
and \eqref{lim1} we get that as $\epsilon\rightarrow 0$,
\stepcounter{limit}
\begin{equation}\label{lim2}
\begin{aligned}
\Big\|\frac{{\rm div}\bv^\epsilon}{\epsilon}\Big\|_1\le&\|\bv^\epsilon\cdot\nabla\eta^\epsilon\|_1
+\|\eta{\rm div}\bv^\epsilon\|_1+\|\epsilon{\rm div}(P^\epsilon(U^\epsilon+\bv^\epsilon))\|_1
+\|U^\epsilon\cdot\nabla \eta^\epsilon\|_1\rightarrow 0.
\end{aligned}\end{equation}

Due to \eqref{lim1} and
$$\frac{\nabla \eta^\epsilon+\nabla\theta^\epsilon}{\epsilon}
=\epsilon F^\epsilon-\bv^\epsilon\cdot\nabla \bv^\epsilon-\theta^\epsilon\nabla\eta^\epsilon
-\eta^\epsilon\nabla\theta^\epsilon-U^\epsilon\cdot\nabla \bv^\epsilon
+\mu\triangle \bv^\epsilon+(\mu+\lambda)\nabla{\rm div}\bv^\epsilon$$
with $F^\epsilon=(\epsilon P^\epsilon+\eta^\epsilon)\f-(\epsilon P^\epsilon
+\eta^\epsilon)(U^\epsilon+\bv^\epsilon)\cdot\nabla(U^\epsilon+\bv^\epsilon)
-\theta^\epsilon\nabla P^\epsilon- P^\epsilon\nabla\theta^\epsilon$,
which comes from the transform of $\eqref{c1}_2$, one deduces, recalling Poincar\'e's inequality, that
\stepcounter{limit}
\begin{equation}\label{lim3}
\begin{aligned}
\Big\|\frac{\eta^\epsilon+\theta^\epsilon}{\epsilon}\Big\|_2\rightarrow0,\ \ \ \ \ \ \ {\rm as}\ \ \epsilon\to 0.
\end{aligned}\end{equation}

On the other hand, in view of Lemma \ref{lmI2}, we observe that $(U^\epsilon,P^\epsilon)$ is a
uniform-in-$\epsilon$ bounded sequence in $(H^4\cap H_0^1)\times\bar{H}^3$. Hence, there are a subsequence of
$(U^{\epsilon_k},P^{\epsilon_k})$, still denoted by $(U^{\epsilon_k},P^{\epsilon_k})$ for simplicity,
and $(\bar U,\bar P)\in(H^4\cap H_0^1)\times\bar{H}^3$, such that as $\epsilon\to 0$,
$$(U^{\epsilon},P^{\epsilon})\rightharpoonup(\bar U,\bar P),\ \ \
\mbox{weakly in }\;(H^4\cap H_0^1)\times\bar{H}^3,$$
and
$$(U^{\epsilon},P^{\epsilon})\rightarrow(\bar U,\bar P),\ \ \ \mbox{strongly in }\;(H^3\cap H_0^1)\times\bar{H}^2.$$

Thus, if we take to the limit as $\epsilon\to 0$ in \eqref{i1} and \eqref{c1}, we conclude
that $(\bar U,\bar P)$ is a solution of the steady incompressible Naiver-Stokes equations (\ref{ISNS}).

In conclusion, we have that
$$\lim_{\epsilon\to 0}\inf_{U,P\in \mathbf{L}}\|U^\epsilon+\bv^\epsilon-U\|_3
+\Big\|P^\epsilon+\frac{\eta^\epsilon+\theta^\epsilon}{\epsilon}-P\Big\|_2+\|\theta^\epsilon\|_3=0,$$
where $\mathbf{L}$ is the same as in Theorem \ref{thm}.
Thus, the proof of the low Mach number limit is completed.
\\[4mm]
\textbf{Acknowledgements.} The authors would like to thank Prof. Qiangchang Ju and Prof. Yue-Jun Peng for
helpful discussions. This work was supported by the National Basic Research Program
under the Grant 2011CB309705, NSFC (Grant Nos. 11229101, 11201115 and 11301083)
and China Postdoctoral Science Foundation (Nos. 2012M510365, 2012M520205).

\vspace{5mm}

\end{document}